\documentclass[11pt]{article}
\usepackage{graphicx}
\usepackage{tikz,pgfplots}
\usepgfplotslibrary{fillbetween}
\usetikzlibrary{calc} 
\usepackage{amsmath}
\usepackage{amsfonts}
\usepackage{amsthm,comment}
\usepackage{authblk}
\usepackage[T1]{fontenc}
\usepackage{url}
\usepackage{color,ulem,cancel}
\usepackage[margin=1in]{geometry}
\usepackage{amssymb,bm}
\usepackage{amsmath}
\usepackage{amsthm}
\usepackage{graphicx}
\usepackage[active]{srcltx} 
\usepackage{hyperref}
\hypersetup{pdfborder=0 0 0}

\newtheorem{theorem}{{\sc Theorem}}[section]
\newtheorem{proposition}[theorem]{{\sc Proposition}}

\newtheorem{lemma}[theorem]{{\sc Lemma}}

\newtheorem{remark}[theorem]{Remark}

\newtheorem{definition}[theorem]{Definition}

\def\XXint#1#2#3{{\setbox0=\hbox{$#1{#2#3}{\int}$ }
\vcenter{\hbox{$#2#3$ }}\kern-.6\wd0}}

\newcommand{\Gth}{\theta}

\bmdefine\BGa{\alpha}
\bmdefine\BGb{\beta}
\bmdefine\BGd{\delta}
\bmdefine\BGe{\epsilon}
\bmdefine\BGve{\varepsilon}
\bmdefine\BGf{\phi}
\bmdefine\BGvf{\varphi}
\bmdefine\BGg{\gamma}
\bmdefine\BGc{\chi}
\bmdefine\BGi{\iota}
\bmdefine\BGk{\kappa}
\bmdefine\BGl{\lambda}
\bmdefine\BGn{\eta}
\bmdefine\BGm{\mu}
\bmdefine\BGv{\nu}
\bmdefine\BGp{\pi}
\bmdefine\BGth{\theta}
\bmdefine\BGvth{\vartheta}
\bmdefine\BGr{\rho}
\bmdefine\BGvr{\varrho}
\bmdefine\BGs{\sigma}
\bmdefine\BGvs{\varsigma}
\bmdefine\BGt{\tau}
\bmdefine\BGj{\tau}
\bmdefine\BGu{\upsilon}
\bmdefine\BGo{\omega}
\bmdefine\BGx{\xi}
\bmdefine\BGy{\psi}
\bmdefine\BGz{\zeta}
\bmdefine\BGD{\Delta}
\bmdefine\BGF{\Phi}
\bmdefine\BGG{\Gamma}
\bmdefine\BGL{\Lambda}
\bmdefine\BGP{\Pi}
\bmdefine\BGT{\Theta}
\bmdefine\BGS{\Sigma}
\bmdefine\BGU{\Upsilon}
\bmdefine\BGO{\Omega}
\bmdefine\BGX{\Xi}
\bmdefine\BGY{\Psi}

\bmdefine\BCA{{\mathcal A}}
\bmdefine\BCB{{\mathcal B}}
\bmdefine\BCC{{\mathcal C}}
\bmdefine\BCD{{\mathcal D}}
\bmdefine\BCE{{\mathcal E}}
\bmdefine\BCF{{\mathcal F}}
\bmdefine\BCG{{\mathcal G}}
\bmdefine\BCH{{\mathcal H}}
\bmdefine\BCI{{\mathcal I}}
\bmdefine\BCJ{{\mathcal J}}
\bmdefine\BCK{{\mathcal K}}
\bmdefine\BCL{{\mathcal L}}
\bmdefine\BCM{{\mathcal M}}
\bmdefine\BCN{{\mathcal N}}
\bmdefine\BCO{{\mathcal O}}
\bmdefine\BCP{{\mathcal P}}
\bmdefine\BCQ{{\mathcal Q}}
\bmdefine\BCR{{\mathcal R}}
\bmdefine\BCS{{\mathcal S}}
\bmdefine\BCT{{\mathcal T}}
\bmdefine\BCU{{\mathcal U}}
\bmdefine\BCV{{\mathcal V}}
\bmdefine\BCW{{\mathcal W}}
\bmdefine\BCX{{\mathcal X}}
\bmdefine\BCY{{\mathcal Y}}
\bmdefine\BCZ{{\mathcal Z}}

\bmdefine\Bzr{ 0}
\bmdefine\Ba{ a}
\bmdefine\Bb{ b}
\bmdefine\Bc{ c}
\bmdefine\Bd{ d}
\bmdefine\Be{ e}
\bmdefine\Bf{ f}
\bmdefine\Bg{ g}
\bmdefine\Bh{ h}
\bmdefine\Bi{ i}
\bmdefine\Bj{ j}
\bmdefine\Bk{ k}
\bmdefine\Bl{ l}
\bmdefine\Bm{ m}
\bmdefine\Bn{ n}
\bmdefine\Bo{ o}
\bmdefine\Bp{ p}
\bmdefine\Bq{ q}
\bmdefine\Br{ r}
\bmdefine\Bs{ s}
\bmdefine\Bt{ t}
\bmdefine\Bu{ u}
\bmdefine\Bv{ v}
\bmdefine\Bw{ w}
\bmdefine\Bx{ x}
\bmdefine\By{ y}
\bmdefine\Bz{ z}
\bmdefine\BA{ A}
\bmdefine\BB{ B}
\bmdefine\BC{ C}
\bmdefine\BD{ D}
\bmdefine\BE{ E}
\bmdefine\BF{ F}
\bmdefine\BG{ G}
\bmdefine\BH{ H}
\bmdefine\BI{ I}
\bmdefine\BJ{ J}
\bmdefine\BK{ K}
\bmdefine\BL{ L}
\bmdefine\BM{ M}
\bmdefine\BN{ N}
\bmdefine\BO{ O}
\bmdefine\BP{ P}
\bmdefine\BQ{ Q}
\bmdefine\BR{ R}
\bmdefine\BS{ S}
\bmdefine\BT{ T}
\bmdefine\BU{ U}
\bmdefine\BV{ V}
\bmdefine\BW{ W}
\bmdefine\BX{ X}
\bmdefine\BY{ Y}
\bmdefine\BZ{ Z}

\everymath{\displaystyle}

\usepackage{xcolor}

\newcommand{\jmsf}[1]{ {   #1} }

\def\bbR{\mathbb{R}}

\pgfplotsset{compat=1.18} 
\begin{document}
\title{Fractional Korn's inequalities without boundary conditions}
\author{D. Harutyunyan\thanks{Department of Mathematics, University of California Santa Barbara, harutyunyan@ucsb.edu}, \ 
T. Mengesha\thanks{Department of Mathematics, The University of Tennessee Knoxville, mengesha@utk.edu}, \ 
H. Mikayelyan\thanks{School of Mathematical Sciences, University of Nottingham Ningbo China, Hayk.Mikayelyan@nottingham.edu.cn}, \  
and J.M. Scott\thanks{ Department of Applied Physics and Applied Mathematics, Columbia University, 
jms2555@columbia.edu}}
\date{}
\maketitle
\begin{abstract}
{Motivated by a linear nonlocal model of elasticity, this work establishes fractional analogues of Korn's first and second inequalities for vector fields in fractional Sobolev spaces defined over a bounded domain.  The validity of the inequalities require no additional boundary condition,  extending existing fractional Korn's inequalities that are only applicable for Sobolev vector fields satisfying zero Dirichlet boundary conditions. The domain of definition is required to have a  $C^{1}$-boundary or, more generally,  a Lipschitz boundary with small Lipschitz constant.  We conjecture that the inequalities remain valid for vector fields defined over any Lipschitz domain. We support this claim  
by presenting a proof of the inequalities for vector fields defined over planar convex domains. }
\end{abstract}

\section{Introduction and main results}
In this paper, we prove the fractional analogues of Korn's first and second inequalities in the so-called bounded Lipschitz domains with small Lipschitz constant. For $n\geq 2$, suppose that $\Omega\subset\mathbb{R}^{n}$ is a bounded domain with Lipschitz
 boundary. 
Classical Korn's inequalities give a means of controlling the $L^{p}$-norm of the gradient $\nabla \Bu$ of a Sobolev vector field $\Bu: \Omega\to \mathbb{R}^{n}$ by the symmetric part of its gradient, $e(\Bu)(\Bx) = \frac{1}{2} (\nabla \Bu (\Bx) + \nabla \Bu ^{\intercal}(\Bx))$ (and the vector field itself). For $1<p<\infty$,  a version of the classical Korn's first inequality \cite{Korn1,Korn2} 
states that there is a constant $C=C(\Omega)>0,$ such that 
\begin{equation}
\label{Intro-Korn1st}\inf_{\BA\in \text{Skew}(\mathbb{R}^{n})}\|\nabla \Bu - \BA\|_{L^{p}(\Omega)} \leq C \|e(\Bu)\|_{L^{p}(\Omega)}\quad \text{for all}\quad \Bu\in W^{1,p}(\Omega, \mathbb{R}^{n}).
\end{equation}
Here, $\text{Skew}(\mathbb{R}^n)$ represents the set of $n\times n$ skew symmetric matrices. 
Korn's second inequality reads as follows  \cite{Korn1,Korn2}: there is a constant $C=C(\Omega)>0,$ such that 
\begin{equation}\label{Intro-Korn2nd}
\|\nabla \Bu \|_{L^{p}(\Omega)} \leq C \left(\|e(\Bu)\|_{L^{p}(\Omega)}  + \|\Bu\|_{L^{p}}\right)\quad \text{for all}\quad \Bu\in W^{1,p}(\Omega, \mathbb{R}^{n}).
\end{equation}
These inequalities play a fundamental role in establishing the well posedness of the linear equations of elastostatics, a system of partial differential equations arising from linearized elasticity under various boundary conditions \cite{Korn1,Korn2,Kond-Oleinik}. 

By a fractional analogue of these inequalities we mean estimates of these type for vector fields in the fractional Sobolev spaces $W^{s,p}(\Omega;\mathbb{R}^{n}),$ for $s\in (0,1)$, where $\Bu\in L^{p}(\Omega;\mathbb{R}^{n})$ is in $W^{s,p}(\Omega;\mathbb{R}^{n})$ if and only if 
\[
|\Bu|_{W^{s,p}(\Omega;\mathbb{R}^{n})}^{p} =  \int_{\Omega}\int_{\Omega}\frac{|\Bu(\Bx)-\Bu(\By)|^p}{|\Bx-\By|^{n+ps}}d\Bx d\By <\infty. 
\]
The function space $W^{s,p}(\Omega;\mathbb{R}^{n})$ is a Banach space with the norm $\|\cdot\|_{W^{s,p}(\Omega)} = \|\cdot\|_{L^{p}(\Omega)}  + |\cdot |_{W^{s,p}(\Omega)}$. When writing inequalities analogous to \eqref{Intro-Korn1st} and \eqref{Intro-Korn2nd} for vector fields in $W^{s,p}(\Omega;\mathbb{R}^{n})$, one must find proper substitutes for the notion of gradient and symmetric part of the gradient. It is intuitively clear that the difference quotient $\frac{\Bu(\By) - \Bu(\Bx)}{|\By-\Bx|}$ could be used as a substitute for the gradient of $\nabla \Bu$ while the seminorm $|\Bu|_{W^{s,p}(\Omega)}$  replaces $\|\nabla \Bu\|_{L^{p}(\Omega)}$. Noting that for $\Bx\in \Omega$ and $\By$ close to $\Bx$, one has for suffficiently smooth vector fields $\Bu(\Bx)$ that
{\small$
\frac{(\Bu(\By) - \Bu(\Bx))}{|\By-\Bx|}\cdot \frac{(\By-\Bx)}{|\By - \Bx|} \approx e(\Bu)(\Bx)\frac{(\By-\Bx)}{|\By-\Bx|}\cdot \frac{(\By-\Bx)}{|\By - \Bx|},$}
we will use the projected difference quotient $\frac{(\Bu(\By) - \Bu(\Bx))}{|\By-\Bx|}\cdot \frac{(\By-\Bx)}{|\By - \Bx|}$ as the nonlocal analogue for $e(\Bu),$ and the seminorm
\[
[\Bu]^{p}_{\mathcal{X}^{s,p}(\Omega)} :=\int_{\Omega}\int_{\Omega}\frac{\left|(\Bu(\By) - \Bu(\Bx))\cdot \frac{(\By-\Bx)}{|\By - \Bx|}\right|^p}{|\By-\Bx|^{d+ps}} d\By d\Bx
\] to replace its norm. This weighted norm of the projected difference quotient not only approximates the norm of $e(\Bu)$  but also inherits its zero sets. Indeed, for $L^1_{loc}(\Omega)$ vector fields $\Bu$, the equality $(\Bu(\Bx)-\Bu(\By))\cdot(\Bx-\By)=0$ holds for a.e. $\Bx,\By\in\Omega,$ if and only if $\Bu$ is an infinitesimal rigid motion \cite[Proposition 1.2]{Temam-Miran}, i.e., $\Bu$ has the form 
\[
\Bu(\Bx) = \BA\Bx + \Bb,\quad\text{for some $\BA\in \text{Skew}(\mathbb{R}^{n})$ and $\Bb\in \mathbb{R}^n$}. 
\]
Those are exactly Sobolev vector fields that make $e(\Bu)= \boldsymbol{0}$  $a.e.$ in $\Omega,$ as can be seen from (\ref{Intro-Korn1st}). We denote this class of vector fields by $\mathcal{R}$. Before we are ready to state the desired fractional analogues of Korn's inequalities, we need to introduce the definition of bounded domains with Lipschitz constant not exceeding a number $L>0.$ 
\begin{definition}
 An open bounded Lipschitz domain 
 $\Omega\subset\mathbb R^n$ is said to have a Lipschitz constant $\leq L,$ if the boundary of $\Omega$ can be covered by finitely many balls (or cylinders) $B_i,$ $i=1,2,\dots,m,$ so that each portion 
$\partial\Omega\cap B_i$ is the graph of a Lipschitz function with Lipschitz constant $\leq L,$ upon a rotation of the coordinate system. 
\end{definition}

\begin{remark}
An open bounded domain {   with} $C^1$-boundary has a local Lipschitz constant as small as any initially chosen positive constant $\epsilon>0.$
\end{remark}

\begin{theorem}\label{Intro:mainthm}
    Let $n \geq 2$, $s \in (0,1)$, $p \in (1,\infty)$. There exists a constant $M_0>0,$ depending only on $n,p,$ and $s,$ such that the following holds: for any open bounded Lipschitz set $\Omega \subset \mathbb{R}^n$ with Lipschitz constant $\leq M_0,$ there exist positive constants $C_1$ and $C_2,$ depending only on $d,s,p$ and $\Omega,$
    such that for all $\Bu \in W^{s,p}(\Omega;\mathbb{R}^n)$, one has
\begin{equation}
\label{thm:korn1st}
\inf_{\Br \in \mathcal{R}}|{\Bu}-{\Br}|^p_{W^{s,p}(\Omega)} \leq C_1 [\Bu]^{p}_{\mathcal{X}^{s,p}(\Omega)},
\end{equation}
and
\begin{equation}
|{\Bu}|^p_{W^{s,p}(\Omega)} \leq C_2([{\Bu}]^p_{\mathcal{X}^{s,p}(\Omega)} + \|\Bu\|^p_{L^{p}(\Omega)}),\label{thm:korn2nd} 
\end{equation}
{  where $\mathcal{R}$ is the class of infinitesimal rigid motions.}

\end{theorem}
Some remarks are in order. The same way the classical Korn's inequalities are linked to the linearized elasticity, so are their fractional analogues to some nonlocal models of elasticity. We discuss here one such model, peridynamics,  a continuum  nonlocal theory of mechanics of materials initially proposed by Stewart Silling \cite{Silling2001,Silling2007,Silling2010}. In bond-based linearized peridynamics, a material occupying a domain $\Omega$ is approximated to be a complex mass-spring system where material points interact, at a distance, with each other over a bond joining them. If the material is subject  to a deformation $\Bv(\Bx) = \Bx + \Bu(\Bx)$, then $\frac{(\Bu(\By) - \Bu(\Bx))}{|\By-\Bx|}\cdot \frac{(\By-\Bx)}{|\By - \Bx|}$ represents a (unit less) linearized nonlocal strain at $\Bx$ along the bond $\boldsymbol{\xi} =\By-\Bx$. The total strain energy is postulated to be proportional to 
\[
W_\rho(\Bu) = \int_{\Omega}\int_{\Omega} \rho(\By-\Bx) \left|\frac{(\Bu(\By) - \Bu(\Bx))}{|\By-\Bx|}\cdot \frac{(\By-\Bx)}{|\By - \Bx|}\right|^{2} d\By d \Bx,
\]
where $\rho(\boldsymbol{\xi})$ is locally integrable and serves as a weight for the long-range interactions.  Given an external force $\Bf \in L^{2}(\Omega;\mathbb{R}^n)$, the corresponding configuration can be found as a minimizer of the functional
\begin{equation}\label{pot-func}
\Bu\mapsto W_\rho(\Bu) - \int_{\Omega} \Bf(\Bx) \cdot \Bu(\Bx) d\Bx
\end{equation}
over an appropriate admissible subset of $L^{2}(\Omega;\mathbb{R}^n)$. In fact, existence of minimizers in some subsets of the energy space 
$
\mathcal{S}_{\rho}(\Omega) = \{\Bu\in L^{2}(\Omega;\mathbb{R}^n): W_\rho(\Bu) <\infty\}
$
is demonstrated in 
\cite{Mengesha-Du}. See also \cite{Mengesha-Du-non} for existence of solutions to more generalized models of linearized peridynamics. Except for $n=1$, the question that whether $\mathcal{S}_{\rho}(\Omega)$, which is based on the {\it projected difference-quotient},  
is equal to the space 
\[ \{\Bu\in L^{2}(\Omega;\mathbb{R}^n): \int_{\Omega}\int_{\Omega} \rho(\By-\Bx) \frac{\left|\Bu(\By) - \Bu(\Bx)\right|^2}{|\By-\Bx|^2} d \By d\Bx < \infty\}\]
based on the full difference-quotient remains open. The fractional Korn's inequalities proved in Theorem \ref{Intro:mainthm} address this question and establish equality of sets for a special case when $\rho(\boldsymbol{\xi}) = |\boldsymbol{\xi}|^{-n-2(s-1)}$, for $s\in (0, 1)$. In this case,  $W_\rho(\Bu) = [\Bu]^{2}_{\mathcal{X}^{s,2}(\Omega)}$ and 
minimizing the functional in \eqref{pot-func} over a weakly closed subset of the smaller $W^{s,2}(\Omega;\mathbb{R}^n)$, say,  $W^{s,2}_{\omega}(\Omega;\mathbb{R}^n)= \{\Bu \in W^{s,2}(\Omega;\mathbb{R}^n): \omega\subset\Omega,\,\, \Bu=\boldsymbol{0},\  \text{a.e. in} \ \omega\}$ is possible. To apply Hilbert space methods, inequality \eqref{thm:korn2nd} is now essential, along with a Poincar\'e-Korn inequality, see Lemma \ref{K-P} below, to show the coercivity of the functional. We leave the details to the interested readers, see \cite{Mengesha-Du, Mengesha-Du-non}.

We emphasize that the main contribution of this work is proving inequalities \eqref{thm:korn1st} and \eqref{thm:korn2nd} for vector field in  $W^{s,p}(\Omega,\mathbb{R}^n)$ without any ``boundary conditions.'' While \eqref{thm:korn1st} 
as stated appears to be new to our best knowledge, its special version with $\BA=0,$ and inequality \eqref{thm:korn2nd} have appeared in recent works, albeit in restricted forms. Indeed, the variant of \eqref{thm:korn1st} with $\BA=0$ was first proven in \cite{M-half} for the case when $\Omega$ is the half-space, $p=2,$ and for vector fields $\Bu\in W_0^{s,2}(\Omega)$ which is the closure of $C^{1}_{c}(\Omega,\mathbb{R}^n)$ with respect to the norm $\|\cdot\|_{W^{s,2}(\Omega)}$ (roughly speaking for vector fields satisfying zero Dirichlet boundary conditions on $\partial\Omega$), see \cite{NPV}. The estimate for the half-space was then extended for any values $1<p<\infty$ in \cite{MScott2018Korn}. The estimate was then proven in \cite{MS2022} for the restricted class $W^{s,p}_0(\Omega,\mathbb{R}^n)$ for bounded $C^1$ domains, and the same result appeared in \cite{Rut.2022} significantly shortening the proof presented in \cite{MS2022}. A tighter version of estimates \eqref{thm:korn1st} and \eqref{thm:korn2nd}  have also been proven in \cite{HM2023} for the case when $ps>1$, where for some constant $C>0,$ 
\begin{align}\label{HM:korn1st}
|{\Bu}|_{W^{s,p}(\Omega)} \leq C [{\Bu}]_{\mathcal{X}^{s,p}(\Omega)}\quad \text{for all} \quad \Bu \in W^{s,p}_0(\Omega;\mathbb{R}^n).
\end{align}
Via  a counterexample \cite{HM2023}, inequality \eqref{HM:korn1st} is shown to fail for any open bounded subset $\Omega\subset\mathbb R^n$ in the case $ps<1$. This is in stark contrast to the case when $\Omega=\mathbb{R}^{n}$ or $\Omega=\mathbb{R}^{n}_{+}$ (unbounded domains), where \eqref{HM:korn1st} is proved to hold for any $s\in(0,1)$, $p\in (1, \infty)$ such that $ps\neq 1,$ \cite{M-half, MS2022}.   In fact, in this case, our current work implies that the restriction $ps\neq 1$ is not even necessary.   
We note that \eqref{HM:korn1st} is the fractional analogue of another version of Korn's first inequality: 
\[
\|\nabla \Bu \|_{L^{p}(\Omega)} \leq C_p \|e(\Bu)\|_{L^{p}(\Omega)}\quad\text{for all} \quad \Bu\in W^{1,p}_0(\Omega, \mathbb{R}^{n}),
\]
for a constant $C_p>0$ that depends only on $p.$ This being said, a new phenomenon occurs in the fractional setting. 
\begin{remark}
It is well known that the range of exponent that validates the classical Korn inequalities \eqref{Intro-Korn1st} and \eqref{Intro-Korn2nd} is $1<p<\infty.$  Moreover, for Sobolev vector fields that satisfy zero Dirichlet boundary conditions, one can always choose $\BA=0$ 
in \eqref{Intro-Korn1st} in that range. However, this is no longer true in the fractional setting because despite the fact that the case $ps<1$ is included in the validity range for \eqref{thm:korn1st} and \eqref{thm:korn2nd}, the version of \eqref{thm:korn1st}
with $\BA=0$ fails in bounded domains in the case 
$ps<1.$
\end{remark}

As it is clear from the formulation, Theorem \ref{Intro:mainthm} has the limitation that inequalities \eqref{thm:korn1st} and \eqref{thm:korn2nd} are established for a class of vector fields defined over a domain with a boundary  that has a sufficiently small Lipschitz constant. Taking clues from the classical Korn's inequalities \cite{nitsche1981korn}, we conjecture that in fact the inequalities remain valid for any bounded Lipschitz domain. To support the claim, we establish the same inequalities for planar convex Lipschitz domains with no constraint on the size of Lipschitz constant of the boundary.  This will be demonstrated in Section \ref{sec:convex}.

As we will show in Section \ref{sec:epigraph}, inequality  \eqref{thm:korn1st} follows from \eqref{thm:korn2nd}. The main challenge is thus proving \eqref{thm:korn2nd}. Our method of proof is standard. We first establish \eqref{thm:korn2nd} for epigraphs supported by a Lipschitz function and then use a partition of unity to localize near the boundary of the domain. The later part of the argument is successfully carried out in \cite{MS2022} and \cite{Rut.2022} and we will not repeat it here. We would rather focus on obtaining the estimate for epigraphs. That will be accomplished after proving the existence
of an extension operator to extend the vector fields in the epigraph to be defined on $\mathbb{R}^n$. As in \cite{MS2022} we will use the extension introduced in \cite{nitsche1981korn} which allows us to control the seminorm of the extended vector fields by the seminorm over the epigraph. In this work, we use an improved Hardy-type inequality, Lemma \ref{lem:2.4} in Section \ref{sec:tools}, to overcome a technical difficulty that we encountered in \cite{MS2022} and restricted the validity \eqref{thm:korn2nd} to only vector fields that vanish on the boundary.

\section{Korn-Poincar\'e and Hardy-type inequalities}
\label{sec:tools}
{  Given an open set $D\subset \mathbb{R}^{n},$ we define the spaces $\mathcal{X}^{s,p}_{0}(D)$ and $\mathcal{X}^{s,p}(D)$ to be the closure of $C_{c}^{1}(D;\mathbb{R}^n)$  and $C^{1}_c(\overline{D};\mathbb{R}^n)$, respectively, with respect to the norm $\|\cdot\|_{\mathcal{X}^{s,p}(D)}, $ where $C_{c}^{1}(\bar{D};\mathbb{R}^n)$ is the set of $C^1$ functions whose support is compactly contained in $\overline{D}$.  
 It is known that  for bounded domains with Lipschitz boundary, $C^{1}(\overline{D};\mathbb{R}^n)$ is  dense in $\mathcal{X}^{s,p}(D)$, as shown in \cite[Theorem 3.3]{Mengesha-Scott-LinLoc}.} 
We begin with the following Korn-Poincar\'e inequalities that are compatible with the seminorm $[\cdot]_{\mathcal{X}^{s,p}(\Omega)}$.  
It is worth mentioning that the fractional Korn-Poincar\'e inequality is an important component in the proof of the first and second fractional Korn inequalities. This is in contrast to the classical local setting where the Korn-Poincar\'e inequality is derived as a consequence of Korn's first inequality after the latter has been established by other means.

\begin{lemma}[Korn-Poincar\'e inequalities]\label{K-P}
Suppose that $\Omega$ is a bounded Lipschitz domain. Then for any $s\in (0,1)$, $p\in (1, \infty)$,  there exists a positive constant $C$ depending only on $n,p,s,$ and $\Omega,$ such that 
\begin{align*}
\min_{\Br\in \mathcal{R}}\|{\bf u}-\Br\|_{L^{p}(\Omega)} &\leq C [{\bf u}]_{\mathcal{X}^{s,p}(\Omega)} \quad \text{for all} \quad {\bf u}\in \mathcal{X}^{s,p}(\Omega).
\end{align*}
Moreover, if $V\subset L^{p}(\Omega;\mathbb{R}^{n})$ is a weakly closed subset such that $V\cap \mathcal{R} = \{\boldsymbol{0}\}$, then there exists a constant $\tilde C>0,$ that in addition may depend on $V,$ such that  
\[
\|\Bu\|_{L^{p}(\Omega;\mathbb{R}^{n})} \leq \tilde C [{\bf u}]_{\mathcal{X}^{s,p}(\Omega)} \quad \text{for all}\quad {\bf u} \in V.
\]
 
\end{lemma}
\begin{proof}
We prove the first assertion. The proof of the second can be found in \cite[Proposition 2.7]{Mengesha-Du-non}. 
We will use a standard contradiction argument adopted by Kondratiev and Oleinik for the classical case in \cite{Kond-Oleinik}. Suppose that there is a sequence  $\Bu_k\in \mathcal{X}^{s,p}(\Omega)$ and the corresponding minimizers  $\BA_k\in \mathrm{skew}(\mathbb R^n)$ and ${\bf b}_k\in\mathbb R^n$ such that 
\begin{equation}
\label{seq-bound}
\|\Bu_k-\BA_k\cdot \Bx-{\bf b}_k\|_{L^{p}(\Omega)}=1\quad \text{and} \quad [\Bu_k]_{\mathcal{X}^{s,p}(\Omega)}\leq 1/k,\quad k=1,2,\dots
\end{equation}
Upon passing to the fields $\Bv_k=\Bu_k-\BA_k\cdot \Bx - {\bf b}_k$ we can assume without loss of generality that $\BA_k=0$ and ${\bf b}_k = 0$ in \eqref{seq-bound} for all $k$. Thus we have the minimality conditions 
\begin{equation}
\label{optimality}
\|\Bv_k\|_{L^{p}(\Omega)}\leq \|\Bv_k-\BA\cdot \Bx-{\bf b}\|_{L^{p}(\Omega)}\quad \text{for any}\quad \BA\in \mathrm{skew}(\mathbb R^n), \ \  {\bf b}\in \mathbb R^n, \ \ k=1,2,\dots
\end{equation}
We then have from \eqref{seq-bound} that the sequence ${\bf v}_k$ is bounded in $\mathcal{X}^{s,p}(\Omega)$. 
We can now apply the compactness theorem in \cite[Theorem~1.3]{DMT-compact}  to conclude that   
the sequence $\{\Bv_k\}$ is pre-compact in $L^p(\Omega),$ thus we can assume without loss of generality that 
\begin{equation}
\label{Lp-convergence}
\Bv_k\to \Bv \quad \text{in}\quad L^p(\Omega),
\end{equation}
for some field $\Bv\in L^p(\Omega).$ We have by  (\ref{seq-bound}) that
\begin{align*}
 \|\Bv_k-\Bv_m\|_{\mathcal{X}^{s,p}(\Omega)}&= ([\Bv_k-\Bv_m]_{\mathcal{X}^{s,p}(\Omega)}+\|\Bv_k-\Bv_m\|_{L^p(\Omega)})\\ \nonumber
 & \leq  C([\Bv_k]_{\mathcal{X}^{s,p}(\Omega)}+[\Bv_m]_{\mathcal{X}^{s,p}(\Omega)}+\|\Bv_k-\Bv_m\|_{L^p(\Omega)})\\ \nonumber
 &\leq C(1/k+1/m+\|\Bv_k-\Bv_m\|_{L^p(\Omega)}),
 \end{align*}
thus the condition (\ref{Lp-convergence}) implies that the sequence $\{\Bv_k\}$ is Cauchy and thus is convergent in $\mathcal{X}^{s,p}(\Omega)$ and the limit is ${\bf v}$. This gives $\Bv_k\to \Bv$ in  $\mathcal{X}^{s,p}(\Omega)$ as $k\to \infty$ as well. 
We thus have from (\ref{seq-bound}) that 
\begin{align*}
[\Bv]_{\mathcal{X}^{s,p}(\Omega)}&\leq [\Bv_k]_{\mathcal{X}^{s,p}(\Omega)}+[\Bv-\Bv_k]_{\mathcal{X}^{s,p}(\Omega)}\\
&\leq 1/k+[\Bv-\Bv_k]_{\mathcal{X}^{s,p}(\Omega)}\to0
\end{align*}
as $k\to\infty,$ thus  $[\Bv]_{\mathcal{X}^{s,p}(\Omega)}=0,$ which gives 
\begin{equation}
\label{rigid-bound}
\Bv(\Bx)=\BA\cdot \Bx+\Bb,\quad\text{for a.e.}\quad\Bx\in\Omega,
\end{equation}
for some constant skew-symmetric matrix $\BA\in \mathbb R^{n\times n}$ and some vector $\Bb\in\mathbb R^n$ (see \cite[Proposition 1.2]{Temam-Miran} or \cite[Theorem 3.1]{Korn-characterization}). 
We then have by (\ref{seq-bound}), (\ref{optimality}), and (\ref{rigid-bound}) that  
\begin{align*}
1=\|\Bv_k\|_{L^p(\Omega)}
 \leq \|\Bv_k-\BA\cdot \Bx-{\bf b}\|_{L^p(\Omega)} 
 = \|\Bv_k-\Bv\|_{L^{p}(\Omega)}\to 0, 
\end{align*}
as $k\to\infty,$ which is a contradiction.
\end{proof}
\begin{remark}
\label{Korn-Poincare-remark}
We remark that if, for a given $\tau>0$ and $\Bx_0 \in \mathbb{R}^n$, $Q(\Bx_0, \tau)$ represents a cube centered at $\Bx_0$ with {   side length $2\tau$}, then a simple scaling argument yields the estimate 
\begin{equation}
\label{Korn-Poincare}
\min_{\Br\in \mathcal{R}}\|{\bf u}-\Br\|_{L^{p}(Q(\Bx_0, \tau))} \leq C\, \tau^{s}\,[{\bf u}]_{\mathcal{X}^{s,p}(Q(\Bx_0, \tau))}
\end{equation}
for all ${\bf u}\in \mathcal{X}^{s,p}(Q(\Bx_0, \tau))$, 
where the constant $C$ is the constant which depends only on $n,$ $p$, and $s$ and the unit cube $Q({\boldsymbol 0}, 1)$.
\end{remark}
The following variant of the fractional Hardy-type 
inequality is key for proving the boundedness of the extension operator we will define in the next section. For notational convenience, we represent points $\Bx\in \mathbb{R}^n$ as $\Bx = (\Bx', x_n)\in \mathbb{R}^{n-1}\times \mathbb{R}$.

\begin{definition}[Epigraph]
Let $f\colon \mathbb R^{n-1}\to\mathbb R$ be a continuous function. 
The set
$$D=\{(\Bx',x_n) \ : \ \Bx'\in \bbR^{n-1},\ x_n>f(\Bx')\}$$ 
is called an epigraph supported by the function $f$. 
In that case we also denote 
$$D_{-}=\{(\Bx',x_n) \ : \ \Bx'\in \bbR^{n-1},\ x_n<f(\Bx')\}.$$ 
\end{definition}

In what follows, $f$ will be a globally Lipschitz function with $\|\nabla f\|_{L^{\infty}(\mathbb R^{n-1})}\leq M.$ Also, capital letter $C$ will denote a constant that depends on $n,p,s$ and $M,$ while small letter $c$ will denote a constant that depends only on $n,p$ and $s.$ For any epigraph $D$ and any $\eta>0$, define the mapping $\Phi_\eta: {D_{-}} \to D$ given by
\begin{equation}
\label{Phi}
\Phi_\eta(\Bx)=(\Bx', f(\Bx') + \eta(f(\Bx')-x_n)),
\end{equation}
which is clearly a Lipschitz diffeomorphism with the inverse 
\begin{equation*}
(\Phi_\eta)^{-1}(\Bx) = (\Bx', f(\Bx') + \frac{1}{\eta}(f(\Bx')-x_n)),
\end{equation*}
and $\text{det}(\nabla \Phi_\eta) = -\eta$.
By direct calculation we get 
\begin{equation}
\label{NormPhi}
\|\nabla \Phi_{\eta}\|_{L^{\infty}(D_-)}=
\sqrt{n-1 + \eta^{2} + (1+\eta)^{2}{  \|\nabla f\|^2_{L^{\infty}}}},
\end{equation}
and
\begin{equation}
\label{NormPhiInverse}
\|\nabla (\Phi_{\eta})^{-1}\|_{L^{\infty}(D)}=  \sqrt{n-1 + \frac{1}{\eta^{2}} + (1+\frac{1}{\eta})^{2}{  \|\nabla f\|^2_{L^{\infty}}}}.
\end{equation}

Hence, in space dimensions $n\geq 2,$ the norms $\|\nabla\Phi_{\eta}\|_{L^{\infty}(D_-)}$ and 
$\|\nabla (\Phi_{\eta})^{-1}\|_{L^{\infty}(D)}$
are bounded from below by one (independent of the Lipschitz constant of $f$). Moreover, as proved in 
Lemma~\ref{lem:A.1}, there exists a constant $C>0,$
depending only on $n,\eta,$ and $M$ such that
\begin{equation}\label{flat-vs-curved}
|\Bx-\By| \leq C|(\Phi_\eta)^{-1}(\Bx) - \By|
\quad \text{for all}\quad \Bx, \By\in D.
\end{equation}

\vspace{0.3cm}

\begin{lemma}[Hardy-type inequality]
\label{lem:2.4}
Let $f\colon \mathbb R^{n-1}\to\mathbb R$ be a Lipschitz function with Lipschitz constant $\leq M,$ and let $D\subset \mathbb R^n$ be the epigraph supported by $f.$ There exist a constant $C=C(n,p,s,M)>0,$ and a constant {   $c_2(n)\geq1$}  (coming from the Whitney cover of $D$), such that for every $\lambda,\mu \in [1-\delta,1+\delta]$ with {  $\delta=\frac{1}{2c_2(n)\sqrt{n}(2+M)}$} and every vector field $\Bu\in \mathcal{X}^{s,p}(D)$ one has
{\small \begin{equation}
\label{Hardy-type}
\int_{D}\frac{|u_n(\Bx',f(\Bx')+\lambda(x_n-f(\Bx')))-u_n(\Bx',f(\Bx')+\mu(x_n-f(\Bx')))|^p}{|x_n-f(\Bx')|^{ps}}d\Bx \leq C[\Bu]_{\mathcal{X}^{s,p}(D)}^p.
\end{equation}}
\end{lemma}

\begin{proof} 
{   Given the epigraph $D$ supported by $f$ as in the assumption of the lemma,  }
we consider the sequence of cubes $\{Q_k\}_{k=1}^\infty$ in $\mathbb{R}^{n}$  with the property that
\begin{itemize}
\item[(i)]{   $D = \cup_{k} Q_k$, and the $Q_k$ are mutually disjoint,  }
\item  [(ii)]the doubled cubes $\hat Q_k=2\cdot Q_k$ satisfy the inclusion $ 
\hat Q_k\subset D$ for all $k\in\mathbb N,$ and that they 
have the finite intersection property
\begin{equation*}
\sum_{k=1}^\infty\chi_{\hat Q_k}(\Bx) \leq c_1(n)\quad\text{for all}\quad \Bx\in D, \quad \text{and }
\end{equation*}
\item[(iii)] there exists a constant $c_2(n)$ such that each of the $c_2(n)-$times enlarged cube $c_2(n)\cdot Q_k$ intersects with the graph of $f.$
\end{itemize}
{   Such a covering of the open set $D$ is called a Whitney cover. Given an open set, it is always possible to construct a Whitney cover for it,  see \cite[Chapter VI, Theorem 1]{Stein} for details.}
In the above, the constants $c_1(n)$ and $c_2(n)$
depend only on the space dimension $n.$  Let now $a_k>0$ be the side length of $Q_k.$ Observe that on one hand condition (i) in particular implies that 
\begin{equation}
\label{2.18}
|x_n-f(\Bx')|\geq \frac{a_k}{2},\quad\text{for all}\quad k\in\mathbb N, \ \Bx\in Q_k. 
\end{equation}
On the other hand for a fixed point $\Bx=(\Bx',x_n)\in Q_k,$ let $\mathrm{dist}(\Bx,\mathrm{Gr}(f))=|\Bx-\By|,$
where $\By=(\By',f(\By'))\in\mathrm{Gr}(f)$ and $\mathrm{Gr}(f)$ is the graph of $f.$ We have that 
$|f(\Bx')-f(\By')|\leq M|\Bx'-\By'|$,
thus we can estimate 
{   
\begin{equation*}
    \begin{split}
        |x_n - f(\Bx')| &= |(\Bx',x_n) - (\Bx',f(\Bx'))| \\
        &\leq |(\Bx',x_n) - (\By',f(\By'))| +  |(\By',f(\By')) - (\Bx',f(\Bx'))| \\
        & =  \mathrm{dist}(\Bx,\mathrm{Gr}(f)) + |(\By',f(\By')) - (\Bx',f(\Bx'))| \\
        &\leq \mathrm{dist}(\Bx,\mathrm{Gr}(f)) + \sqrt{1+M^2}|\By'-\Bx'| \\
        & \leq (1+ \sqrt{1+M^2}) \, \mathrm{dist}(\Bx,\mathrm{Gr}(f))\,.
    \end{split}
\end{equation*}
From condition (iii) we have $\mathrm{dist}(\Bx,\mathrm{Gr}(f)) \leq \frac{c_2(n)+1}{2} \sqrt{n}a_k \leq c_2(n) \sqrt{n} a_k$, hence
\begin{equation}
\label{2.19} 
|x_n-f(\Bx')|\leq c_2(n) \sqrt{n} (2+M) a_k \quad\text{for all}\quad k\in\mathbb N, \ \Bx\in Q_k. 
\end{equation} 
}

{   By the definition of $\mathcal{X}^{s,p}(D),$ it suffices to take $\Bu\in C^{1}_c(\overline{D}, \mathbb{R}^n)$. }
 We can then estimate 
\begin{align*}
\int_{D}&\frac{|u_n(\Bx',f(\Bx')+\lambda(x_n-f(\Bx')))-u_n(\Bx',f(\Bx')+\mu(x_n-f(\Bx')))|^p}{|x_n-f(\Bx')|^{ps}}d\Bx \\ \nonumber
&\leq \sum_{k=1}^\infty
\int_{Q_k}\frac{|u_n(\Bx',f(\Bx')+\lambda(x_n-f(\Bx')))-u_n(\Bx',f(\Bx')+\mu(x_n-f(\Bx')))|^p}{|x_n-f(\Bx')|^{ps}}d\Bx. 
\end{align*}
Setting 
$\Phi_\lambda^\ast(\Bx)=(\Bx',f(\Bx')+\lambda(x_n-f(\Bx')))$ for brevity, we aim to prove the inequality 
\begin{equation}
\label{2.20}
\int_{Q_k}\frac{|u_n(\Phi_\lambda^\ast(\Bx))-u_n(\Phi_{\mu}^\ast(\Bx))|^p}{|x_n-f(\Bx')|^{ps}}d\Bx 
\leq C [\Bu]_{\mathcal{X}^{s,p}(\hat Q_k)}^p,
\end{equation}
for each cube $Q_k.$ For every fixed $k\in\mathbb N$ we have by (\ref{2.18}), that
\begin{equation}
\label{2.21}
\int_{Q_k}\frac{|u_n(\Phi_\lambda^\ast(\Bx))-u_n(\Phi_{\mu}^\ast(\Bx))|^p}{|x_n-f(\Bx')|^{ps}}d\Bx
\leq {   2^{ps}}a_k^{-ps}\int_{Q_k} |u_n(\Phi_\lambda^\ast(\Bx))-u_n(\Phi_{\mu}^\ast(\Bx))|^p d\Bx.
\end{equation}
Next we apply the Korn-Poincar\'e inequality (\ref{Korn-Poincare}) to the cube 
$\hat Q_k$ and the vector field $\Bu.$ Hence, there exists a constant $C>0,$ a skew-symmetric matrix $\BA_k\in\mathbb R^{n\times n},$ and a vector $\Bb_k\in\mathbb R^n$ such that 
\begin{equation}
\label{2.22}
\|\Bu(\cdot)-\BA_k(\cdot)-\Bb_k\|_{L^p(\hat Q_k)}\leq
 C a_k^{s}[\Bu]_{\mathcal{X}^{s,p}(\hat Q_k)}.
\end{equation} 
Observe that $\Phi_\lambda^\ast \colon Q_k\to \Phi_\lambda^\ast(Q_k)$ is a one-to-one diffeomorphism with the inverse $\Phi_{1/\lambda}^\ast$ and has Jacobian equal to $\lambda\leq 1+\delta.$ Also, due to the inequality (\ref{2.19}), we have that, for every 
$\Bx\in Q_k$ and every $\lambda\in [1-\delta,1+\delta],$ 
\begin{align*}
|\Phi_\lambda^\ast(\Bx)-\Bx|=|1-\lambda||x_n-f(\Bx')|
\leq \delta c_2(n)\sqrt{n}(2+M)a_k 
\leq \frac{a_k}{2},
\end{align*}
provided $\delta=\frac{1}{2c_2(n)\sqrt{n}(2+M)}.$ This implies the   
inclusion conditions $\Phi_\lambda^\ast(Q_k),\Phi_\mu^\ast(Q_k)\subset \hat Q_k.$ Consequently, {   noting that by skew-symmetry 
$(\BA_k\cdot\Phi_\lambda^\ast(\Bx)+\Bb_k)_n
 -(\BA_k\cdot\Phi_{\mu}^\ast(\Bx)+\Bb_k)_n=0$, and using the bound in (\ref{2.22}),  we can estimate that} 

\begin{align}
\label{2.23}
 \int_{Q_k} &|u_n(\Phi_\lambda^\ast (\Bx))-u_n(\Phi_{\mu}^\ast(\Bx))|^p d\Bx \\ \nonumber
 &= \int_{Q_k} |u_n(\Phi_\lambda^\ast (\Bx))-(\BA_k\cdot\Phi_\lambda^\ast(\Bx)+\Bb_k)_n
 -(u_n(\Phi_{\mu}^\ast (\Bx))-(\BA_k\cdot\Phi_{\mu}^\ast(\Bx)+\Bb_k)_n)|^p d\Bx \\ \nonumber
 &\leq 2^{p-1}\int_{Q_k} |u_n(\Phi_\lambda^\ast (\Bx))-(\BA_k\cdot\Phi_\lambda^\ast(\Bx)+\Bb_k)_n|^pd\Bx \\ \nonumber 
&+2^{p-1}\int_{Q_k} |u_n(\Phi_{\mu}^\ast (\Bx))-(\BA_k\cdot\Phi_{\mu}^\ast(\Bx)+\Bb_k)_n|^p d\Bx \\ \nonumber 
&\leq C\|\Bu(\By)-\BA_k\cdot\By-\Bb_k\|_{L^p(\hat Q_k)}^p \\ \nonumber
&\leq C a_k^{ps}[\Bu(\By)]_{\mathcal{X}^{s,p}(\hat Q_k)}^p.
\end{align}
Putting together now (\ref{2.21}) and (\ref{2.23}) we discover 
\begin{equation*}
\int_{Q_k}\frac{|u_n(\Phi_\lambda^\ast(\Bx))-
u_n(\Phi_{\mu}^\ast(\Bx))|^p}{|x_n-f(\Bx')|^{ps}}d\Bx
\leq C[\Bu(\By)]_{\mathcal{X}^{s,p}(\hat Q_k)}^p.
\end{equation*}
In order to complete the proof of the lemma, one needs to sum (\ref{2.20}) over $k$ and keep in mind the finite intersection property in (ii). This completes the proof of Lemma~2.4.
\end{proof}
\begin{remark}\label{remark-Hardy}
    In the special case of the half space, where $D= \mathbb{R}^{n}_{+}$ and $f\equiv 0$, inequality \eqref{Hardy-type} reduces to 
    \[
\int_{D}\frac{|u_n(\Bx',\lambda x_n)-u_n(\Bx',\mu x_n)|^p}{|x_n|^{ps}}d\Bx \leq C[\Bu]_{\mathcal{X}^{s,p}(D)}^p.
\]
This inequality was proved in {\cite[Lemma 4.1]{M-half}} for vector fields in $\mathcal{X}^{s,p}_{0}(D)$ {   for particular values of $\lambda$ and $\mu$} under the extra assumption that $ps\neq1.$ It is now clear from Lemma \ref{lem:2.4} that this requirement is not necessary and that the fractional Korn's inequality proved in \cite{M-half} for vector fields in $\mathcal{X}^{s,p}_{0}(D)$ is also valid for all $s\in (0,1)$ and $1<p<\infty$ (even when $ps=1$). A consequence of this is that the Korn inequality proved in \cite{MS2022} for vector fields defined on bounded domains with smooth boundary will also be true for the full ranges of $s$ and $p$.  

{   We note that we refer { to} the inequality \eqref{Hardy-type} {as} a Hardy-type because the inequality captures the optimal decay rate to zero of  a 
map near the boundary, say in the case when $D= \mathbb{R}^{n}_{+}$, $(\Bx', x_n)\mapsto u_n(\Bx', \lambda x_n) - u_n(\Bx', \mu x_n)$, which vanishes on the hyperplane $\partial D = \{x_n=0\}$, in terms of an appropriate seminorm.  See \cite{Dyda,Loss-Sloane} for the standard fractional Hardy-type inequalities. }
\end{remark}

\section{Fractional Korn's inequalities}
\label{sec:epigraph}
\subsection{Korn's second inequality over epigraphs}

This section is devoted to the fractional Korn's second inequality {  for vector fields defined over} epigraphs. We prove the following theorem. 
\begin{theorem}[Korn's second inequality in epigraphs]
\label{thm:Korn.Epigraph}
Given $s\in (0,1)$ and $1 < p<\infty$, there exists a universal constant $M_0>0$ and another constant $C_0>0$ depending only on $n, p,s$ and $M_0$ with the following property: For any epigraph $D$ supported by {   $f:\mathbb{R}^{n-1}\to \mathbb{R}$} with $\|\nabla f\|_{L^{\infty}}\leq M_0$, one has for all {   $\Bu\in \mathcal{X}^{s,p}(D)$} the inequality
{   \begin{equation}
\label{Korn.sec.epi}
|\Bu|_{W^{s,p}(D)} \leq C_0[\Bu]_{\mathcal{X}^{s,p}(D)}.
\end{equation}}
\end{theorem}
As we described in the introduction, to prove the fractional Korn's inequality (\ref{Korn.sec.epi}) for an epigraph $D$, we first prove the existence of an extension operator to extend the vector fields in $\mathcal{X}^{s,p}(D)$ to be in $\mathcal{X}^{s,p}(\mathbb{R}^n)$ in such a way that the seminorm of the extended vector fields is controlled by the seminorm over $D$. As in \cite{MS2022}, we will show that the extension operator that was used in \cite{nitsche1981korn} for the proof of the classical Korn's inequality will also be useful to prove the fractional case. 
\begin{proposition}[Extension operator]
\label{prop:extension}
Let $s\in (0,1)$ and $1\leq p<\infty$ and let $D$ be an epigraph supported by a Lipschitz function {   $f:\mathbb{R}^{n-1}\to \mathbb{R}$} with 
$\|\nabla f\|_{L^{\infty}}=M.$ There exists a bounded extension operator $E: W^{s,p}(D;\mathbb{R}^n)\to W^{s, p}(\mathbb{R}^n;\mathbb{R}^n),$ a constant $C>0,$ depending only on $n, p, s,$ and $M,$ and a constant $c>0$ depending only on $n, p,$ and $s,$ with the property that for all $\Bu\in W^{s,p}(D;\mathbb{R}^n)$ one has
{  \begin{equation}\label{ext-norm-estimate}
[E(\Bu)]_{\mathcal{X}^{s,p}(\mathbb{R}^n)} \leq C[\Bu]_{\mathcal{X}^{s,p}(D)} + c(1+M)^{2+\frac{n}{2p}} M|\Bu|_{W^{s,p}(D)}.
\end{equation} }
\end{proposition}

\begin{proof}
By density {   of $C_c^{1}(\overline{D};\mathbb{R}^n)$ in $W^{s,p}(D;\mathbb{R}^{n})$ (see \cite[Theorem 6.70]{Leoni-2023})}, it suffices to show the inequality for $C^1$ vector fields. Following the approach in \cite{nitsche1981korn}, we define the extension operator $E$ as follows. 
For $\Bu = (\Bu',u_n) \in C_c^1(\overline{D};\mathbb{R}^n)$, and for constants $\lambda$, $\mu$, $k$, $\ell$, $m$, and $q$, set
\begin{equation}\label{ext-epigraph}
[ \mathrm{E}(\Bu)(\Bx) ]_i :=
\begin{cases}
u_i(\Bx)\,, & \qquad \Bx \in D\,, \quad i = 1, 2, \ldots n-1, n\,, \\
k \, u_i^{\lambda}(\Bx) + \ell \, u_i^{\mu}(\Bx)\,, & \qquad \Bx \in D_-\,, \quad i = 1, 2, \ldots n-1\,, \\
m \, u_n^{\lambda}(\Bx) + q \, u_n^{\mu}(\Bx)\,, & \qquad \Bx \in D_-\,,
\end{cases}
\end{equation}
where
\begin{equation}\label{ulambdamu}
\begin{split}
u_j^{\lambda}(\Bx) := u_j \big( \Bx',f(\Bx') + \lambda(f(\Bx')-x_d) \big) = u_j(\Phi_\lambda(\Bx)),\\
u_j^{\mu}(\Bx) := u_j \big( \Bx',f(\Bx') + \mu(f(\Bx')-x_d) \big)\,=u_j(\Phi_\mu(\Bx)).
\end{split}
\end{equation}
We choose constants $\lambda$, $\mu$, $k$, $\ell$, $m$, $q$, such that
\begin{equation}\label{eq:ReflectionConditions}
\begin{split}
\lambda > 0\,, \quad \mu > 0\,, \qquad k+ \ell = 1 = m+q\,, \qquad \lambda k = -m\,, \quad \mu \ell = -q\,.
\end{split}
\end{equation}
For $0\leq \lambda<\mu$ these constants are uniquely 
defined and are given by 
\begin{equation}\label{klmq}
k=\frac{1+\mu}{\mu-\lambda},\quad \ell=-\frac{1+\lambda}{\mu-\lambda},\quad
m=-\frac{\lambda(1+\mu)}{\mu-\lambda},\quad q=\frac{\mu(1+\lambda)}{\mu-\lambda}.
\end{equation}
Let now $\delta=\frac{1}{2c_2(n)\sqrt{n}(1+M)}$ be as in Lemma~2.4, and choose 
\begin{equation}
\label{Def.LambdaMu}
\lambda=1-\delta\quad\text{and}\quad \mu=1+\delta, 
\end{equation}
{   where we note that since  $\delta \leq 1/2,$ $\lambda, \mu\in \left[1/2, 3/2\right]$, and $M>0$. }
Recalling that the boundary $\partial D$ is given by the equation $x_n=f(\Bx')$, it is clear that the operator $E$ takes continuous map defined on $D$ to continuous maps on $\mathbb{R}^n$. Moreover, for {   $\Bu\in C_c^{1}(\overline{D},\mathbb{R}^n)$, $E(\Bu) \in W^{s,p}(\mathbb{R}^{n}, \mathbb{R}^{n})$. This can be shown  following calculations similar to the ones that will be used below estimating $[E(\Bu)]_{\mathcal{X}^{s,p}(\mathbb{R}^{n})}$ to demonstrate the inequality \eqref{ext-norm-estimate}. } We split the domain of integration and write 
\begin{align}\label{Int-split}
[E(\Bu)]^p_{\mathcal{X}^{s,p}(\mathbb{R}^{n})}  = [\Bu]^p_{\mathcal{X}^{s,p} (D)}& + [E(\Bu)]^p_{\mathcal{X}^{s,p}(D_{-})}\nonumber\\
&+2\int_{D_{-}}\int_{D} \frac{\left|(E(\Bu)(\Bx) - E(\Bu)(\By)) \cdot(\Bx-\By)\right|^p}{|\Bx-\By|^{n + (s+1)p}}d\By d \Bx.
\end{align}
We need to estimate the second and the third terms. For $\Bx\in D_{-}$, we write $$E(\Bu)(\Bx)= E_{\lambda}(\Bu)(\Bx)  + E_{\mu}(\Bu)(\Bx) $$ 
where $E_{\lambda}(\Bu)(\Bx) = (k (\Bu')^{\lambda}, m\, u^{\lambda}_{n})$ and $E_{\mu}(\Bu)(\Bx) = (\ell (\Bu')^{\mu}, q\, u^{\mu}_{n})$. We have  $$[E(\Bu)]^p_{\mathcal{X}^{s,p}(D_{-})} \leq 2^{p-1}([E_\lambda(\Bu)]^p_{\mathcal{X}^{s,p}(D_{-})}+[E_\mu(\Bu)]^p_{\mathcal{X}^{s,p}(D_{-})}),$$
and will estimate each of the summands next. To estimate $[E_\lambda(\Bu)]^p_{\mathcal{X}^{s,p}(D_{-})}$, we make the change of coordinates $\Bz=\Phi_\lambda(\Bx)$ and $\Bw=\Phi_\lambda(\By)$ and recall the discussion about the mapping $\Phi$ in (\ref{Phi})--(\ref{NormPhiInverse}) to write the integral as  
{\small \begin{align*}
&\lambda ^2\, [E_\lambda(\Bu)]^p_{\mathcal{X}^{s,p}(D_{-})}\\
& = \int_{D}\int_{D}\frac{|k(\Bu'(\Bz) - \Bu'(\Bw))\cdot (\Bz'-\Bw') + m(u_n(\Bz)-u_n(\Bw))\cdot([(\Phi_\lambda)^{-1}(\Bz)]_n - [(\Phi_\lambda)^{-1}(\Bw)]_n)|^{p}}{|(\Phi_\lambda)^{-1}(\Bz) - (\Phi_\lambda)^{-1}(\Bw)|^{n+ (s+1)p}}d\Bz d\Bw.
\end{align*}}

Notice that 
$$[(\Phi_\lambda)^{-1}(\Bz)]_n - [(\Phi_\lambda)^{-1}(\Bw)]_n = -\frac{1}{\lambda}(z_n - w_n) + \frac{1+\lambda}{\lambda}(f(\Bz')-f(\Bw')),$$
and  
\[
|\Bz - \Bw| \leq \|\nabla \Phi_{\lambda}\|_{L^{\infty}(D_-)} \left| (\Phi_{\lambda})^{-1}(\Bz) - (\Phi_{\lambda})^{-1}(\Bw) \right|. 
\]
It then follows using the relation $\lambda \,k = -m,$ that
\begin{align*}
    [E_\lambda(&\Bu)]^p_{\mathcal{X}^{s,p}(D_{-})} \leq \frac{2^{p-1}k^p }{\lambda^2} \|\nabla \Phi_{\lambda}\|^{n+(s+1)p}_{L^{\infty}(D_-)}\,[{\Bu}]^p_{\mathcal{X}^{s,p}(D)}\\ \nonumber
    &+ \frac{2^{p-1}m^p(1+ \lambda)^p}{\lambda^{2+p}} \|\nabla \Phi_{\lambda}\|^{n+(s+1)p}_{L^{\infty}(D_-)}\int_{D}\int_{D}{\frac{\left| \big( u_n(\Bz) - u_n(\Bw) \big) \cdot \big( f(\Bz')-f(\Bw') \big) \right|^p}{|\Bz-\Bw|^{d+(s+1)p}}} d{\Bw}d{\Bz}\\ \nonumber
    &\leq \frac{2^{p-1}k^p }{\lambda^2} \|\nabla \Phi_{\lambda}\|^{n+(s+1)p}_{L^{\infty}(D_-)}\,\left([{\Bu}]^p_{\mathcal{X}^{s,p}(D)} + (1+ \lambda)^p \|\nabla f\|^{p}_{L^{\infty}}|u_n|^p_{W^{s,p}(D)}\right).
\end{align*}
A similar estimate as above also holds for $[E_\mu(\Bu)]^p_{\mathcal{X}^{s,p}(D_{-})}$ where $\lambda, k$ are replaced by $\mu$ and $\ell$. {   We now combine the two estimates, keeping in mind (\ref{NormPhi}), from which we have $\|\nabla \Phi_{\lambda}\|^{n+(s+1)p}_{L^{\infty}(D_-)} \leq C { (1 + M)^{n + 2p} }$, 
and the explicit formulae (\ref{klmq}) and (\ref{Def.LambdaMu}) which imply $k^p \leq C(1 + M)^{p}$, to obtain the bound, after some calculations, that  }
\begin{equation*}
 [E(\Bu)]^p_{\mathcal{X}^{s,p}(D_{-})} \leq C[{\Bu}]^p_{\mathcal{X}^{s,p}(D)}+c(1+M)^{\frac{n}{2}+2p}
 M^{p}|\Bu|^p_{W^{s,p}(D)}.
\end{equation*}
It remains is to estimate the third term $\int_{D_{-}}\int_{D} \frac{\left|(E(\Bu)(\Bx) - E(\Bu)(\By)) \cdot(\Bx-\By)\right|^p}{|\Bx-\By|^{n + (s+1)p}}d\By d \Bx$ in \eqref{Int-split}. To that end, we denote the integral by $I_{mix}$ and for $\Bx\in D_{-}$ write 
\[E({\Bu})(\Bx) = k\Bu(\Phi_\lambda(\Bx)) + \ell\Bu(\Phi_\mu(\Bx)) + (m-k)(\boldsymbol{0}', u_n(\Phi_\lambda(\Bx))) + (q-\ell)(\boldsymbol{0}', u_n(\Phi_{\mu}(\Bx))), \]
It then follows by algebraic calculations and using the relations { (\ref{eq:ReflectionConditions})
and (\ref{klmq}) }between $k, \ell, m$, and $q$ that for $\Bx\in D_{-}$ and $\By\in D$: 
\begin{equation*}
\begin{split}
&(E(\Bu)(\Bx) - E(\Bu)(\By)) \cdot(\Bx-\By)\\
&=
{
k \big( \Bu(\Phi_{\lambda}(\Bx)) - \Bu(\By) \big) \cdot \big( \Bx - \By \big) + 
\ell \big( \Bu(\Phi_{\mu}(\Bx)) - \Bu(\By) \big) \cdot \big( \Bx - \By \big)}
\\
&\quad+{ 
(m-k)  \big( u_n(\Phi_{\lambda}(\Bx)) - u_n(\Phi_{\mu}(\Bx)) \big) \cdot \big( x_n - y_n \big)
}
\\
&=k \big( \Bu(\Phi_{\lambda}(\Bx)) - \Bu(\By) \big) \cdot \big( \Phi_{\lambda}(\Bx) - \By \big) + \ell \big( \Bu(\Phi_{\mu}(\Bx)) - \Bu(\By) \big) \cdot \big( \Phi_{\mu}(\Bx) - \By \big) \\
&\quad + (k-m)  \big( u_n(\Phi_{\lambda}(\Bx)) - u_n(\Phi_{\mu}(\Bx)) \big) \cdot\big( y_n - f(\Bx') \big)
\\
&{\quad\begin{split}    +
k \big( u_n(\Phi_{\lambda}(\Bx)) - u_n(y) \big)
\big( x_n - f(\Bx')-\lambda (f(\Bx') -x_n ) \big) \\ 
 +\ell \big( u_n(\Phi_{\mu}(\Bx)) - u_n(y) \big)
\big( x_n - f(\Bx')-\mu (f(\Bx') -x_n ) \big) \\
 +(k-m) \big( u_n(\Phi_{\lambda}(\Bx)) - u_n(\Phi_{\mu}(\Bx)) \big) \big(  f(\Bx') -x_n \big).
\end{split}}
\end{split}
\end{equation*}
The latter three terms add up to zero. 
We then have the  estimate that 
\[
\begin{split}
    I_{mix}&\leq C k^{p} \int_{D_{-}}\int_{D} \frac{\left|\big( \Bu(\Phi_{\lambda}(\Bx)) - \Bu(\By) \big) \cdot \big( \Phi_{\lambda}(\Bx) - \By \big)\right|^p}{|\Bx-\By|^{n + (s+1)p}}d\By d \Bx\\
    &+ C \ell^{p} \int_{D_{-}}\int_{D} \frac{\left|\big( \Bu(\Phi_{\mu}(\Bx)) - \Bu(\By) \big) \cdot \big( \Phi_{\mu}(\Bx) - \By \big)\right|^p}{|\Bx-\By|^{n + (s+1)p}}d\By d \Bx\\
    &+C|k-m|^{p} \int_{D_{-}}\int_{D} \frac{\left|\big( y_n - f(\Bx') \big) \cdot \big( u_n(\Phi_{\lambda}(\Bx)) - u_n(\Phi_{\mu}(\Bx)) \big)\right|^p}{|\Bx-\By|^{n + (s+1)p}}d\By d \Bx\\
    &= I^{1}_{mix}+ I_{mix}^2 +I_{mix}^3.  
\end{split}
\]
The first two terms $I_{mix}^1$ and $I_{mix}^2$ can be estimated in similar ways. To demonstrate, making  the change of variables $\Bz= \Phi_\lambda (\Bx)$, we obtain that 
\[
I_{mix}^1 = \frac{Ck^p}{\lambda}\int_{D}\int_{D} \frac{\left|\big( \Bu(\Bz) - \Bu(\By) \big) \cdot \big( \Bz - \By \big)\right|^p}{|\left(\Phi_\lambda\right)^{-1}(\Bz)-\By|^{n + (s+1)p}}d\By d \Bz.
\]
We now use \eqref{flat-vs-curved} to estimate the latter by $C [\Bu]_{\mathcal{X}^{s,p}(D)}$. We finish the proof by estimating $I^{3}_{mix}$. Making the variable change $\Bz= \Phi_{1}(\Bx)$ to work solely in $D,$ we have that 
\begin{align*}
I^{3}_{mix} &\leq C \int_{D}\int_{D}\frac{|y_n- f(\Bz')|^p |u_{n}(\Phi_{\lambda}((\Phi_1)^{-1}(\Bz)))-u_{n}(\Phi_{\mu}((\Phi_1)^{-1}(\Bz)))|^p}{|(\Phi_1)^{-1}(\Bz)-\By|^{n + (s+1)p}} d\By d\Bz\\
&= C \int_{D} J(\Bz)| u_{n}(\Bz', f(\Bz')+\lambda(z_n-f(\Bz')))-u_{n}(\Bz', f(\Bz')+\mu(z_n-f(\Bz')))|^{p} d\Bz, 
\end{align*}
where for each $\Bz\in D$ 
\[
J(\Bz) = \int_{D} \frac{|y_n-f(\Bz')|^p}{(|\Bz'-\By'|^2 + |(y_n-f(\Bz')) + (z_n - f(\Bz'))|^2)^{\frac{n+(s+1)p}{2}}} d\By \leq \frac{C}{|z_n - f(\Bz')|^{sp}}
\]
as shown in {    
Lemma \ref{lem:A.2} in the appendix (or \cite[Lemma~\ref{lem:A.1}]{MS2022}).}  As a consequence, we have that 
\begin{equation}
\label{I_mix^3}
I^{3}_{mix}\leq C  \int_{D} \frac{| u_{n}(\Bz', f(\Bz')+\lambda(z_n-f(\Bz')))-u_{n}(\Bz', f(\Bz')+\mu(z_n-f(\Bz')))|^{p}}{|z_n - f(\Bz')|^{sp}} d\Bz.
\end{equation}
Finally, an application of Lemma~\ref{lem:2.4} together with (\ref{I_mix^3}) completes the proof of the proposition.
\end{proof}

We are now ready to prove Theorem~\ref{thm:Korn.Epigraph}.

\begin{proof}[Proof of Theorem~\ref{thm:Korn.Epigraph}]
Theorem~\ref{thm:Korn.Epigraph} follows by an application of the above proposition and Korn's second inequality for $\mathbb R^n,$ \cite{MScott2018Korn}. 
Indeed, let $\Bu \in C^1_c(\overline{D};\mathbb{R}^n)$. As remarked in the proof, $E(\Bu) \in W^{s,p}(\mathbb{R}^n,\mathbb{R}^n)$. Then by the fractional Korn's second inequality 
proved in \cite[Theorem~1.1]{MScott2018Korn} for vector fields defined on $\mathbb{R}^n$, 
we have, on the one hand, for a constant $c_0=c_0(n,p,s),$ that 
\jmsf{
\begin{align*}
|\Bu|_{W^{s,p}(D)}\leq  |E(\Bu)|_{W^{s,p}(\mathbb{R}^n)}
\leq  c_0 |E(\Bu)|_{\mathcal{X}^{s,p}(\mathbb{R}^n)}.
\end{align*}
}
On the other hand, Proposition \ref{prop:extension} yields for constants $C,c>0$ that
\jmsf{
\begin{equation*}
|E(\Bu)|_{\mathcal{X}^{s,p}(\mathbb{R}^n)}\leq C|\Bu|_{\mathcal{X}^{s,p}(D)} + c(1+\|\nabla f\|_{L^{\infty}})^{2+\frac{n}{2p}}\|\nabla f\|_{L^{\infty}} |\Bu|_{W^{s,p}(D)}).
\end{equation*}
}
Consequently, we obtain 
\begin{equation*}
|\Bu|_{W^{s,p}(D)}\leq  c_0C|\Bu|_{\mathcal{X}^{s,p}(D)} + cc_0(1+\|\nabla f\|_{L^{\infty}})^{2+\frac{n}{2p}}\|\nabla f\|_{L^{\infty}} |\Bu|_{W^{s,p}(D)},
\end{equation*}
which yields (\ref{Korn.sec.epi}) \jmsf{for $\Bu \in C^1_c(\overline{D};\mathbb{R}^n)$ provided }
$M_0$ fulfills $(1+M_0)^{2+\frac{n}{2p}}M_0<\frac{1}{cc_0}.$ 
\jmsf{
To conclude that (\ref{Korn.sec.epi}) holds for general $\Bu \in \mathcal{X}^{s,p}(D)$, we use the definition of $\mathcal{X}^{s,p}(D)$. Take a sequence $\{\Bu_j \} \subset C^1_c(\overline{D};\mathbb{R}^n)$ converging to $\Bu$ in $\mathcal{X}^{s,p}(D)$.
Then a subsequence (not relabeled) converges a.e.\ on $D$ and so by Fatou's lemma
\begin{equation*}
    \begin{split}
        |\Bu|_{W^{s,p}(D)} \leq \liminf_{j \to \infty} |\Bu_j|_{W^{s,p}(D)} \leq C_0 \liminf_{j \to \infty} |\Bu_j|_{\mathcal{X}^{s,p}(D)} = C_0 |\Bu|_{\mathcal{X}^{s,p}(D)}.
    \end{split}
\end{equation*}
}
\end{proof}
{  
\begin{remark}
As a consequence of Theorem \ref{thm:Korn.Epigraph}, the extension $E$ in Proposition \ref{prop:extension} is a continuous operator from $\mathcal{X}^{s,p}(D)$ to $\mathcal{X}^{s,p}(\mathbb{R}^n)$, since \eqref{ext-norm-estimate} and \eqref{Korn.sec.epi} yield
\begin{equation*}
    [E(\Bu)]_{\mathcal{X}^{s,p}(\mathbb{R}^n)} \leq C[\Bu]_{\mathcal{X}^{s,p}(D)} + c(1+M)^{2+\frac{n}{2p}} M|\Bu|_{W^{s,p}(D)} \leq C |\Bu|_{\mathcal{X}^{s,p}(D)}.
\end{equation*}
\end{remark}
}

\subsection{Fractional Korn's second inequality in bounded domains}
In this section we provide a proof of inequality (\ref{thm:korn2nd}) in Theorem~\ref{Intro:mainthm}. As already mentioned, we will adopt a partition of unity argument employed in \cite{MS2022}. For the convenience of the reader we repeat the arguments here. Before we present the proof, we make  the following observations related to estimates involving the product $\psi \Bu$ of $\Bu\in \mathcal{X}^{s,p}(\Omega)$ and $\psi\in W^{1, \infty}(\Omega)$. First, such a product $\psi \Bu$ belongs to $ \mathcal{X}^{s,p}(\Omega)$ with the estimate 
\begin{align*}
    [\psi \Bu]_{\mathcal{X}^{s,p}(\Omega)} \leq c\,\|\psi\|_{W^{1, \infty}} \| \Bu\|_{\mathcal{X}^{s,p}(\Omega)}, 
\end{align*}
where $c$ depends only on $n, p, s$, and $ \text{diam}(\Omega)$. This is precisely \cite[Lemma 3.1]{MS2022}. 
Second, due to \cite[Lemma 3.2]{MS2022}, if $\Omega\subset \tilde{\Omega}$, and there exists $\beta>0$ such that for all $\By\in \tilde{\Omega}\setminus \Omega$
\[
\text{dist}(\By, \text{supp($\psi$)}) \geq \beta >0,
\]
then after extending the product by $\boldsymbol{0}$ on $\tilde{\Omega}\setminus \Omega$, it will belong to $\mathcal{X}^{s,p}(\tilde{\Omega})$ with the similar estimate 
\begin{align}\label{prod-zero-ext}
[\psi \Bu]_{\mathcal{X}^{s,p}(\tilde{\Omega})} 
\leq c_1(\beta)\|\psi\|_{W^{1,\infty}}\|\Bu\|_{\mathcal{X}^{s,p}(\Omega)},
\end{align}
where $c_1(\beta)$ depends only on $n, p, s, \text{diam}(\Omega)$, and  $\beta.$ Both statements can be proven by a direct evaluation of the $\mathcal{X}$ seminorm of the product $\psi \Bu,$ see 
\cite{MS2022} for details. We are now ready to present the proof of the the theorem. 

\begin{proof}[Proof of inequality \eqref{thm:korn2nd} of Theorem~  \ref{Intro:mainthm}] 
Let $M_0$ be the constant found in Theorem \ref{thm:Korn.Epigraph}. Suppose that $\Omega$ is a bounded Lipschitz domain with local Lipschitz constant $\leq M_0$. By definition, we may choose an open set $\Omega_0\Subset \Omega$ and open balls $B_{r_j}(\By_j)$, for $j=1, \dots, N$ with centers $\By_j\in \partial \Omega$  such that 
\begin{enumerate}
\item $\Omega=\cup_{j=0}^{N}\Omega_j$ where $\Omega_j = \Omega\cap B_{r_j}(\By_j)$  for $j=1, \dots, N$.

\item For every $1\leq j\leq N$, define $T_j: B_{r_{j}}(\By_j) \to \mathbb{R}^{n}$ to be the operator consisting of the translation $\By_j \to 0$ and a rotation such that $T_j(\partial \Omega\cap B_{r_j}(\By_j) )$ coincides with part of the graph of a Lipschitz function $f_j:\mathbb{R}^{n-1} \to \mathbb{R}$ with $\|\nabla f_j\|_{L^{\infty}(\mathbb R^{n-1})} \leq M_0$. Note that the function $f_j$ is initially only defined on {   an open bounded subset of} $\mathbb R^{n-1},$ but we extend it into all of $\mathbb R^{n-1}$ by Kirszbraun's theorem \cite{Kirzbraun1934}, preserving the Lipschitz constant. This is necessary for the reduction of the situation to epigraphs in $\mathbb R^{n-1}.$
\end{enumerate}
Set $Q_j = T_j (B_{r_j}(\By_j)),$ and also define
\[\begin{split}
Q_j^+ := \{ \Bx \in Q_j \, : \, x_n > f_j(\Bx') \}\,, &\qquad
Q_j^- := \{ \Bx \in Q_j \, : \, x_n < f_j(\Bx') \}\,, \\
K_j^+ := \{ \Bx \in \mathbb{R}^n \, : \, x_n > f_j(\Bx') \}\,, &\qquad
K_j^- := \{ \Bx \in \mathbb{R}^n \, : \, x_n < f_j(\Bx') \}\,.
\end{split}\]
We may choose the map $T_j$ so that  $T_j(\Omega_j) = Q_{j}^{+}$. Note that $T_j$ is a bi-Lipschitz map with Lipschitz constant depending only on $n$ and $\Omega$.  Let $\{ \varphi_j \}_{j=0}^N \subset C^{\infty}_c(\mathbb{R}^n;\mathbb{R})$ be a $C^{\infty}$ partition of unity subordinate to the collection $\{\Omega_0\}\cup \{ B_{r_j}(\By_j) \}_{j=1}^N$. Then for every $1\leq j\leq N,$ we have $\text{supp}(\varphi_j) \subset B_{r_j}(\By_j),$  $\text{dist}( \By, \text{supp}(\varphi_j)) \geq \beta_j > 0$ for every $\By \in \Omega \setminus \overline{\Omega_j} $. We also have $\text{supp} (\varphi_0)\subset \Omega_0$, $\text{dist}( \By, \text{supp}(\varphi_0)) \geq \beta_0 > 0$ for every 
$\By \in \Omega \setminus \overline{\Omega_0},$ and $\textstyle \sum_{j=0}^N \varphi_j \equiv 1$ on $\Omega$. 

Suppose now $\Bu\in C^{1}(\overline{\Omega};\mathbb{R}^n)$. Define $\Bu_j := \varphi_j \Bu$, for $j=0, 1, \dots, N.$ We consider $\Bu_0$ first. After extending it by $\boldsymbol{0}$ to $\mathbb{R}^n$, we have that $\Bu_0\in \mathcal{X}^{s,p}(\mathbb{R}^n)$ and that by the fractional Korn's inequality on $\mathbb{R}^n$,  \cite{MScott2018Korn} and \eqref{prod-zero-ext}
\begin{align}
\label{est-for-u0}
|\Bu_0|_{W^{s,p}(\Omega)} &\leq c[\Bu_0]_{\mathcal{X}^{s,p}(\mathbb{R}^n)}\\ \nonumber
&\leq cc^0_{1}(\beta_0)\|\varphi_0\|_{W^{1,\infty}} \|\Bu\|_{\mathcal{X}^{s,p}(\Omega)}.
\end{align}
For $j=1,\dots,N$ applying again \eqref{prod-zero-ext} using the semi norm $|\cdot|_{W^{s,p}(\Omega)}$ instead of $[\cdot]_{\mathcal{X}^{s,p}(\Omega)}$ we have  
\begin{equation}\label{est-for-ujs}
|\Bu_j|_{W^{s,p}(\Omega)} \leq  c_1^j(\beta_j)\|\varphi_j\|_{W^{1,\infty}}
\|\Bu_j\|_{W^{s,p}(\Omega_j)}
\end{equation}
where $c_1^j(\beta_j)$ depends only on $s,p,n,$ $\text{diam} (\Omega)$ and $\beta_j.$ Now since $T_j$ consists of a rotation and a translation, $\nabla T_j$ is a constant rotation, with $T_j(\Bx) - T_j(\By) = (\nabla T_j) (\Bx-\By)$. Therefore, writing $R_j := \nabla T_j$, define $\Bv_j(\Bx) := R_j \Bu_j(T_j^{-1}(\Bx))$. Then we have $\Bv_j \in W^{s,p}(Q_j^+)$ and that for each $\By\in K^{+}_j\setminus Q_j^{+},$ $\text{dist}(\By, Supp(\Bv_j))\geq \tilde{\beta}_j>0$ for some positive constant $\tilde \beta_j$. Moreover, 
\begin{equation*}
\|\Bv_j\|_{L^{p}(Q_j^{+})} = \|\Bu_j\|_{L^{p}(\Omega_j)},\,\|\Bv_j\|_{W^{s,p}(Q_j^{+})} = \|\Bu_j\|_{W^{s,p}(\Omega_j)},\,\text{and } [\Bv_j]_{\mathcal{X}^{s,p}(Q_j^+)}=[\Bu_j]_{\mathcal{X}^{s,p}(\Omega_j)}.
\end{equation*}
We will demonstrate the last equality as the others can be established similarly. By a change of coordinates,
\begin{equation*}
\begin{split}
[\Bv_j]_{\mathcal{X}^{s,p}(Q_j^+)}^p &= \int_{\Omega_j}\int_{\Omega_j}{ \frac{\left| \big( R_j \Bu_j(\Bx) - R_j \Bu_j(\By) \big) \cdot \big( T_j(\Bx) - T_j(\By) \big)  \right|^p}{|T_j(\Bx) - T_j(\By)|^{n+sp+p}} }d{\By}d{\Bx} \\
	&=  \int_{\Omega_j}\int_{\Omega_j}{ \frac{\left| \big( R_j \Bu_j(\Bx) - R_j \Bu_j(\By) \big) \cdot \big( R_j \Bx - R_j \By \big)  \right|^p}{|R_j \Bx - R_j \By|^{n+sp+p}} }d{\By}d{\Bx} \\
	&=\int_{\Omega_j}\int_{\Omega_j}{ \frac{\left|R_j^{\intercal} R_j \big( \Bu_j(\Bx)-\Bu_j(\By) \big) \cdot \big( \Bx-\By \big)  \right|^p}{|\Bx - \By|^{n+sp+p}} }d{\By}d{\Bx}\\
	&= [\Bu_j]_{\mathcal{X}^{s,p}(\Omega_j)}^p.
\end{split}
\end{equation*}
Extending $\Bv_j$ by $\boldsymbol{0}$ on $K^{+}_j\setminus Q_j^{+}$, we have that $\Bv_j\in C^{1}(K_j^{+};\mathbb{R}^n)$. Applying the fractional Korn's inequality for epigraphs, Theorem \ref{thm:Korn.Epigraph}, we have 
\begin{align} 
\label{Q-to-K}
|\Bv_j|_{W^{s,p}(Q_j^{+})} &\leq |\Bv_j|_{W^{s,p}(K_j^{+})}\\ \nonumber
&\leq C\|\Bv_j\|_{\mathcal{X}^{s,p}(K^{+}_j)},
\end{align}
where $C$ only depends on $s, p, n$ and $M_0$. We may also apply \eqref{prod-zero-ext} to estimate further as 
\begin{equation}
\label{est-rotated-u} 
\|\Bv_j\|_{\mathcal{X}^{s,p}(K^{+}_j)} \leq c_2^j(\tilde{\beta_j})
\|\Bv_j\|_{\mathcal{X}^{s,p}(Q^{+}_j)} 
\end{equation}
We combine now \eqref{est-for-ujs}, \eqref{Q-to-K}, and \eqref{est-rotated-u} to obtain 
\begin{equation}\label{est-for-uj-by-uj}
|\Bu_j|_{W^{s,p}(\Omega)} \leq 
C \|\Bu\|_{\mathcal{X}^{s,p}(\Omega)}, \qquad j=1, 2, \dots, N,
\end{equation}
where $C$ is a positive constant that depends on $s,p,n, \text{diam}(\Omega)$, $M_0$, and the partition of unity. Therefore by \eqref{est-for-u0} and \eqref{est-for-uj-by-uj}, we have 
\[
[\Bu]_{W^{s,p}(\Omega)} = \Big[\sum_{j=0}^{N} \Bu_j\Big]_{W^{s,p}(\Omega)} \leq \sum_{j=0}^{n}[\Bu_j]_{W^{s,p}(\Omega)} \leq C\|\Bu\|_{\mathcal{X}^{s,p}(\Omega)}.
\]
The estimate for vector fields  $\Bu$ in $W^{s,p}(\Omega;\mathbb{R}^n)$ follows by density. This completes the proof. 
\end{proof}

\subsection{Fractional Korn's first inequality in bounded domains}
In this section we provide a proof of inequality (\ref{thm:korn1st}) in Theorem~\ref{Intro:mainthm}.   
\begin{proof}[Proof of inequality \eqref{thm:korn1st} of Theorem~\ref{Intro:mainthm}]
Assume in contradiction (\ref{thm:korn1st}) fails to hold. It then follows that there exist a sequence $\Bu_k\in W^{s,p}(\Omega,\mathbb R^n)$ and a sequence of skew-symmetric matrices $\BA_k\in \mathrm{skew}(\mathbb R^n)$ such that 
\begin{equation}
\label{3.38}
|\Bu_k-\BA_k\cdot \Bx|_{W^{s,p}(\Omega)}{  = \min_{\BA\in \mathrm{skew}(\mathbb R^n)}|\Bu_k-\BA\cdot \Bx|_{W^{s,p}(\Omega)}}=1\,\, \text{ and } \,\, [\Bu_k]_{\mathcal{X}^{s,p}(\Omega)}\leq \frac{1}{k},\quad k=1,2,\dots
\end{equation}
We may also assume that for each $k$ the average of ${\Bu}_k-\BA_k\cdot \Bx$ over $\Omega$ is $\boldsymbol{0}$ by shifting it by a vector 
$\Bb_k\in \mathbb R^n$ if necessary. Upon passing to the fields $\Bv_k=\Bu_k-\BA_k\cdot \Bx-\Bb_k,$ we can further assume without loss of generality that $\BA_k=0$ and $\Bb_k=0$.  Thus  the minimality conditions 
\begin{equation}
\label{3.39}
|\Bv_k|_{W^{s,p}(\Omega)}\leq |\Bv_k-\BA\cdot \Bx|_{W^{s,p}(\Omega)}\quad \text{for any}\quad \BA\in \mathrm{skew}(\mathbb R^n), \ k=1,2,\dots
\end{equation}
hold and by Poincar\'e's inequality, the sequence $\Bv_k$ is bounded in $W^{s,p}(\Omega;\mathbb{R}^n)$. From the compactness theorem,  
\cite[Theorem~7.1]{NPV},  
the sequence $\{\Bv_k\}$ is pre-compact in $L^p(\Omega),$ thus we can assume without loss of generality that 
\begin{equation}
\label{3.40}
\Bv_k\to \Bv \quad \text{in}\quad L^p(\Omega),
\end{equation}
for some field $\Bv\in L^p(\Omega).$ We then have by Korn's second inequality (\ref{thm:korn2nd}) and \eqref{3.39} that
\begin{align*}
 \|\Bv_k-\Bv_m\|_{W^{s,p}(\Omega)}&\leq C([\Bv_k-\Bv_m]_{\mathcal{X}^{s,p}(\Omega)}+\|\Bv_k-\Bv_m\|_{L^p(\Omega)})\\ \nonumber
 & \leq  C([\Bv_k]_{\mathcal{X}^{s,p}(\Omega)}+[\Bv_m]_{\mathcal{X}^{s,p}(\Omega)}+\|\Bv_k-\Bv_m\|_{L^p(\Omega)})\\ \nonumber
 &\leq C(1/k+1/m+\|\Bv_k-\Bv_m\|_{L^p(\Omega)}),
 \end{align*}
thus the condition \eqref{3.40} implies that the sequence $\{\Bv_k\}$ is Cauchy and thus is convergent in $W^{s,p}(\Omega).$ This gives, {   as $k\to \infty$}
\begin{equation}
\label{3.42}
\Bv_k\to \Bv \quad \text{in}\quad W^{s,p}(\Omega).
\end{equation}
From \eqref{3.38} and \eqref{3.42} we have
\begin{align*}
[\Bv]_{\mathcal{X}^{s,p}(\Omega)}&\leq [\Bv_k]_{\mathcal{X}^{s,p}(\Omega)}+[\Bv-\Bv_k]_{\mathcal{X}^{s,p}(\Omega)}\\
&\leq 1/k+[\Bv-\Bv_k]_{W^{s,p}(\Omega)}\to0
\end{align*}
as $k\to\infty,$ thus  $[\Bv]_{\mathcal{X}^{s,p}(\Omega)}=0,$ which gives 
\begin{equation}
\label{3.43}
\Bv(\Bx)=\BA\cdot \Bx+\Bb,\quad\text{for a.e.}\quad\Bx\in\Omega,
\end{equation}
for some constant skew-symmetric matrix $\BA\in \mathbb R^{n\times n}$ and some vector $\Bb\in\mathbb R^n$ \cite[Proposition 1.2]{Temam-Miran}. 
Note that then we have by \eqref{3.38}, \eqref{3.39},  \eqref{3.42}, and \eqref{3.43}:
\begin{align*}
1&=|\Bv_k|_{W^{s,p}(\Omega)}
 \leq |\Bv_k-\BA\cdot \Bx|_{W^{s,p}(\Omega)}
 = |\Bv_k-\Bv|_{W^{s,p}(\Omega)}\to 0 
\end{align*}
as $k\to\infty,$ which is a contradiction. 
\end{proof}

\section{Fractional Korn's inequality for planar polygonal convex domains}\label{sec:convex}
As we discussed in the introduction, we conjecture that the smallness of the Lipschitz constant of the boundary of the domain is not necessary for the validity of the Fractional Korn's inequalities. In this section, we will support this hypothesis by demonstrating the validity of the inequality 
in the case of planar polygonal convex domains. 
The argument of the proof mimics the strategy we used for smooth domains. We begin by proving the inequality for angular domains.  We then cover the boundary of the convex polygonal domain by balls centered on the boundary. The resulting intersecting sets are either wedges (bounded angular domains) or half balls over which  we will have the appropriate estimates. Finally, we use a partition of unity argument to obtain the estimates over the convex polygon. 
In this section, vectors defined on the planar domains are represented as $\Bu = (u_1, u_2)$.

\setcounter{equation}{0}
\subsection{The case of angular domains}
Consider an angular planar domain $D$ with an angle of span in the interval $(0,\pi)$. Upon an affine change of variables, we may  assume without loss of generality that 
$D$ is given by 
\begin{equation}
\label{4.1}
D=\{\Bx\in\mathbb R^2 \ : \ 0<x_1,\ \ \alpha x_1<x_2\},
\end{equation}
for some $\alpha\in\mathbb{R.}$
Note that $D$ is exactly half of the epigraph supported by the function $f(x_1) = \alpha x_1$ defined over $(0,\infty)$.   
In that case, we set 
\begin{equation*}
D_{-}=\{\Bx\in\mathbb R^2 \ : \ 0<x_1,\ \ x_2<\alpha x_1\}.
\end{equation*}
Notice that $\overline{D\cup D_{-}} = \mathbb{R}^{2}_{x_{1} \geq 0 } = \{(x_1, x_2)\in \mathbb{R}^{2}: x_1 \geq 0\}.$ We begin by demonstrating the existence of an extension operator to prove that vector fields defined over $D$ can be extended to $\mathbb{R}^{2}_{x_1 > 0}$ accompanied with an appropriate control of their nonlocal norm.  We use the extension operator defined in \cite{nitsche1981korn} for planar angular domains where it is shown to map $W^{1,2}(D;\mathbb{R}^2)$ to $W^{1,2}(\mathbb{R}^{2}, \mathbb{R}^{2})$. 
\begin{proposition}\label{prop:extesionforangle}
Let $s\in (0, 1)$, $1<p<\infty$, and let $D$ be given by (\ref{4.1}). Then, there exists a bounded extension operator $E: {  \mathcal{X}^{s,p}(D)\to \mathcal{X}^{s, p}( \mathbb{R}^2_{x_1 > 0})}$ such that  $E(\Bu)(\Bx)= \Bu(\Bx)$ for $\Bx\in D$. Moreover, there exists a constant $C>0$ depending only on $p, s,$ and $\alpha,$ such that   for all $\Bu\in \mathcal{X}^{s,p}(D),$
\begin{equation}
\label{4.3}
[E(\Bu)]_{\mathcal{X}^{s,p}(\mathbb{R}^{2}_{x_1 > 0})}\leq C ([\Bu]_{\mathcal{X}^{s,p}(D)}+\|\Bu\|_{L^p(D)}).
  \end{equation}
   
\end{proposition}

\begin{proof}
{   As before, it suffices to prove the inequality for $\Bu\in C^{1}_c(\overline{D};\mathbb{R}^2)$.} 
Following \cite{nitsche1981korn}, we set   $E(\Bu)(\Bx) = \tilde{E}_{\lambda} (\Bu)(\Bx) + \tilde{E}_{\mu}(\Bu)(\Bx)$, for $\Bx\in D_{-}$, where 
\begin{align}\label{3.4}
\tilde{E}_{\lambda} (\Bu)(\Bx) &= (k u_1^\lambda(\Bx) + \alpha k (1+\lambda) u_2^\lambda(\Bx), m u_2^\lambda(\Bx))\nonumber\\ 
\tilde{E}_{\mu} (\Bu)(\Bx) &= (\ell u_1^\mu(\Bx) - \alpha k (1+\lambda) u_2^\mu(\Bx), q u_2^\mu(\Bx)),
\end{align} and $E(\Bu)(\Bx) = \Bu(\Bx)$, if $\Bx\in D$. 
The constants $\lambda$, $\mu$, $k$, $\ell$, $m$, $q$, satisfy the constraints \eqref{eq:ReflectionConditions}-(\ref{Def.LambdaMu}), and the functions $u_i^{\lambda}$ and $u_i^{\mu}$ are defined as before in \eqref{ulambdamu}. Note that this is the extension for epigraphs with the additional summand $\alpha k(1+\lambda)(u_2^\lambda(\Bx)-u_2^\mu(\Bx))$ in the 1st component of 
$E(\Bu)(\Bx)$ for $\Bx\in D_{-}$.  The proof of the estimate in \eqref{4.3} follows the calculations done for the case of the epigraphs. Below we sketch the proof only including those calculations that are new. As before, we begin by decomposing the integral as  
\begin{align}
\label{angle-decomp}
[E(\Bu)]^p_{\mathcal{X}^{s,p}(\mathbb{R}^{2}_{x_1 > 0})}= [E(\Bu)]^p_{\mathcal{X}^{s,p}(D)} &+  [E(\Bu)]^p_{\mathcal{X}^{s,p}(D_{-})}\nonumber\\
&+ 2\int_{D_{-}}\int_{D} \frac{\left|(E(\Bu)(\Bx) - E(\Bu)(\By)) \cdot(\Bx-\By)\right|^p}{|\Bx-\By|^{n + (s+1)p}}d\By d \Bx.
\end{align}
We {  need} to estimate the last two terms. Clearly, 
\[
[E(\Bu)]^p_{\mathcal{X}^{s,p}(D_{-})} \leq 2^{p-1}\left([\tilde{E}_{\lambda}(\Bu)]^p_{\mathcal{X}^{s,p}(D_{-})}  +[\tilde{E}_{\mu}(\Bu)]^p_{\mathcal{X}^{s,p}(D_{-})}\right).
\]
A simple calculation reveals that the additional summands $\alpha k(1+\lambda)u_2^\lambda(\Bx)$ and $\alpha k(1+\lambda)u_2^\mu(\Bx)$ make it possible to simplify further. Indeed, after change of variables $\Bz= \Phi_{\lambda}(\Bx)=(x_1, \alpha x_1 + \lambda(\alpha x_1 - x_2))$  and $\Bw= (y_1, \alpha y_1 + \lambda(\alpha y_1 - y_2))$, we have that
{  \[
(\tilde{E}_{\lambda}(\Bu)(\Bz) - \tilde{E}_{\lambda}(\Bu)(\Bw))\cdot (\Phi_{\lambda}^{-1}(\Bz) - \Phi_{\lambda}^{-1}(\Bw)) = k(\Bu(\Bz) - \Bu(\Bw))\cdot(\Bz-\Bw),
\] and hence }
\[
[\tilde{E}_{\lambda}(\Bu)]^p_{\mathcal{X}^{s,p}(D_{-})} \leq C(\alpha, \lambda, k, p, s) [\Bu]^{p}_{\mathcal{X}^{s,p}(D)}.
\]
Similar estimates also holds for $[\tilde{E}_{\mu}(\Bu)]^p_{\mathcal{X}^{s,p}(D_{-})}$, after noting that the relations between the parameters in \eqref{eq:ReflectionConditions}, implies that $\alpha k(1 + \lambda) = -\alpha \ell(1 + \mu)$. The point here is that 
the additional summand $\alpha \,k(1+\lambda)(u_2^\lambda(\Bx)-u_2^\mu(\Bx))$ in the first component of the extension facilitates a cancellation of the extra term, {   which is $|\Bu|_{W^{s,p}(D)}$ multiplied by the Lipschitz constant $\alpha$,} that would appear if we otherwise use the extension operator \eqref{ext-epigraph} treating the domain as a Lipschitz domain. {   This eliminates the need for the Lipschitz constant to be small so as to absorb the term involving $|\Bu|_{W^{s,p}(D)}$.} What is left now is estimating the mixed integral $\int_{D_{-}}\int_{D}\dots d\By d\Bx$ appearing in \eqref{angle-decomp}. This can be estimated as in the proof of Proposition \ref{prop:extension}. The only difference is that there will be an additional term due to the new term 
$\alpha k(1+\lambda)(u_2^\lambda(\Bx)-u_2^\mu(\Bx))$. This amounts to estimating the expression 
\begin{equation}
\label{4.6} 
\mathrm{II_{new}}=\int_{D_{-}}\int_D\frac{|(u_2^\lambda(\Bx)-u_2^\mu(\Bx))\cdot(x_1-y_1)|^p}{|\Bx-\By|^{2+p+ps}}d\By d\Bx
\end{equation} 
in terms of the norm of ${\Bu}$ in $\mathcal{X}^{s,p}(D)$. 
To prove \eqref{4.6}, by the change of variable  $\Bz=\Phi_1(\Bx),$ we have 
\begin{align*}
\mathrm{II_{new}}&=C\int_{D}\int_D\frac{|(u_2^\lambda(\Phi_1^{-1}(\Bz))-u_2^\mu(\Phi_1^{-1}(\Bz)))\cdot(z_1-y_1)|^p}{|\Phi_1^{-1}(\Bz)-\By|^{2+p+ps}}d\By d\Bz\\ \nonumber
&\leq C\int_D I(\Bz)|u_2^\lambda(\Phi_1^{-1}(\Bz))-u_2^\mu(\Phi_1^{-1}(\Bz))|^p d\Bz,
\end{align*} 
where for any fixed $\Bz\in D$ we have set
\[
I(\Bz)=\int_D \frac{|z_1-y_1|^p}{|\Phi_1^{-1}(\Bz)-\By|^{2+p+ps}} d\By.
\] 
Using Lemma \ref{lem:A.2} from the appendix we have that for each $\Bz\in D,$  
\begin{equation*}
I(\Bz) \leq \frac{C}{|z_2-f(z_1)|^{ps}}, \end{equation*} 
with a constant $C$ that depends only on $p,s,$ and the Lipschitz constant of $f(z_1) = \alpha z_1$, which is $\alpha$ in this case. 
As a consequence, we have 
\begin{align}
\label{4.9} 
\mathrm{II_{new}} &\leq C \int_D \frac{|u_2^\lambda(\Phi_1^{-1}(\Bz))-u_2^\mu(\Phi_1^{-1}(\Bz))|^p}{|z_2-f(z_1)|^{ps}}d\Bz\\ \nonumber
 &= C \int_D \frac{|u_2(z_1,f(z_1)+\lambda(z_2-f(z_1)))-u_2(z_1,f(z_1)+\mu(z_2-f(z_1)))|^p}{|z_2-f(z_1)|^{ps}}d\Bz. 
 \end{align} 
In order to finish the proof we need to estimate the expression in (\ref{4.9}) by the seminorm  $[\Bu]_{\mathcal{X}^{s,p}(D)}.$ This would be straightforward by Lemma~\ref{lem:2.4}, if $D$ was an epigraph (but $D$ is just part of an epigraph). We demonstrate below how the proof of Lemma~\ref{lem:2.4} can be adjusted to this situation. To that end, we need to provide an appropriate Whitney-type cover of $D$. Let 
$F=\{\Bx\in\mathbb R^2 \ : \ x_1,x_2>0\}$ be the first quadrant in $\mathbb R^2.$ We cover $F$ by horizontal rows of identical dyadic cubes as follows: Cover the strip $F\cap\{2^{k}\leq x_2\leq 2^{k+1}\}$ by closed cubes, $Q_k$,  of side length $2^k,$ for every $k\in\mathbb Z$ starting from the $x_2-$axis. {The resulting cover is exactly the restriction of the Whitney cover of the upper half-space on the first quadrant. Notice here that, $Q_k$ is $2^{k}$ distant away from the $x_1$-axis, and the doubled cubes $\hat{Q}^{+}_k$ in the direction of the positive axes have a finite intersection property.} Now, the domain $D$ is the image of $F$ under the bi-Lipschitz mapping 
$$
\varphi : F \to D \text{ defined by } \varphi(x_1,x_2)=(x_1,x_2{+}\alpha x_1).
$$
Each of the dyadic cubes $Q_k$ (from the covering of $F$) will get mapped to a parallelogram $P_k$ which will constitute a Whitney-type cover of {$D,$} by a sequence of dyadic parallelograms. {It is not difficult to see that $P_k$ is a translation of $2^{k}$ times the base parallelogram $\bar{P}_0$ determined by the points $(0,0),$ $(0, 1)$, $(1, \alpha),$ and $(1, 1+\alpha)$. This construction gives rise to a perfect cover of $D,$ as the parallelograms are essentially disjoint. Moreover, for any $k$, the height of parallelogram $P_k$ is comparable to its distant away from the line $x_2 = \alpha x_1 = f(x_1)$, and the finite intersection property of enlarged cubes of the initial Whitney cover will also persist under the mapping $\varphi$. We denote the image of the doubled cubes ${\hat Q}^{+}_k$ by ${\hat P}^{+}_k$. That is, ${\hat P}^{+}_k=\varphi({\hat Q}^{+}_k )$ and, from the construction, these are just translations of $2^k$ times $\bar{P}^{+}_0$, which is the paralellogram determined by the points $(0, 0)$, $(0, 2)$, $(2, 2\alpha)$, $(2, 2  + 2\alpha)$. With this at hand, we can now repeat the argument in the proof of Lemma \ref{lem:2.4}. Since the argument is almost the same for this construction, we only demonstrate  
the analogue of the inequality (\ref{2.23}). To that end,  we have  
\begin{align}
\label{4.7}
 \int_{P_k} &|u_2(\underbrace{z_1,f(z_1)+\lambda(z_2-f(z_1))}_{=\Phi_\lambda^\ast (\Bz)})-
 u_2(\underbrace{z_1,f(z_1)+\mu(z_2-f(z_1))}_{=\Phi_\mu^\ast (\Bz)})|^p
 d\Bz \\ \nonumber
 &\leq 2^{p-1}\int_{P_k} |u_2(\Phi_\lambda^\ast (\Bz))-(\tilde{\BA}_k\cdot\Phi_\lambda^\ast(\Bz)+\Bb_k)_2|^pd\Bz \\ \nonumber 
&+2^{p-1}\int_{P_k} |u_2(\Phi_{\mu}^\ast (\Bz))-(\tilde{\BA}_k\cdot\Phi_{\mu}^\ast(\Bz)+\Bb_k)_2|^p d\Bz \\ \nonumber 
&\leq C\|\Bu(\By)-\tilde{\BA}_k\cdot\By-\Bb_k\|_{L^p(\hat{P}^{+}_k)}^p \\ \nonumber
&\leq C \tilde{a}_k^{ps}[\Bu(\By)]_{\mathcal{X}^{s,p}(\hat{P}^{+}_k)}^p,
\end{align}
where $\tilde{a}_k$ is the height of $P_{k}$ and, as before, we can show that for appropriately chosen $\lambda$ and $\mu$, depending on $\alpha$ and $n$, $\Phi^*_\lambda(P_k), \Phi^*_\mu(P_k)\subset{\hat P}^{+}_k$. Notice that the choice of the infinitesimal rigid displacement $\Bx\mapsto \tilde{\BA}_k\cdot \Bx + \Bb_k$ as well as the last inequality follow from a version of Poincar\'e-Korn inequality over the parallelogram $\hat{P}^{+}_k$ (see Remark \ref{Korn-Poincare-remark}). Indeed, after noting that the area of the base parallelogram $\bar{P}_0$ is 1, then by a simple scaling we have that for any $\tau > 0,$
\[
\min_{\Br\in \mathcal{R}}\|\Bu - \Br\|_{L^{p}(\tau \bar{P}_{0})} \leq C \tau^{s} [\Bu]_{\mathcal{X}^{s,p}(\tau \bar{P}_0)},
\]
where $C$ independent of $\tau.$
Putting together the analogue of (\ref{2.21}) and (\ref{4.7}) we obtain that 
\begin{equation*}
\int_{P_k}\frac{|u_2(\Phi_\lambda^\ast(\Bx))-
u_2(\Phi_{\mu}^\ast(\Bx))|^p}{|x_2-f(x_1)|^{ps}}d\Bx
\leq C[\Bu(\By)]_{\mathcal{X}^{s,p}(\hat P_k)}^p.
\end{equation*}
The rest is similar to the proof of Lemma~\ref{lem:2.4}.
}

\end{proof}

\begin{remark}
Following the above procedure, we can show that the above extension operator is also bounded from $W^{s,p}(D,\mathbb R^2)$ to $W^{s,p}(\mathbb{R}^{2}_{x_1 > 0},\mathbb{R}^2)$. The proposition also implies the fractional Korn's second inequality for planar angular domains. {   
Indeed, let $\Bu\in C^{1}_{c}(\overline{D},\mathbb{R}^2)$. Then by Proposition \ref{prop:extesionforangle}, we can extend $\Bu$  to $E(\Bu)\in \mathcal{X}^{s,p}(\mathbb{R}^{2}_{x_1 > 0})$ 
such that 
\[
[E(\Bu)]_{\mathcal{X}^{s,p}(\mathbb{R}^{2}_{x_1 > 0})} \leq C([\Bu]_{\mathcal{X}^{s,p}(D)} + \|\Bu\|_{L^{p}(D)}).
\]
Noting that $E(\Bu)$ is defined on an epigraph, up to a rotation, we may apply the fractional Korn's inequality for epigraphs, Theorem \ref{thm:Korn.Epigraph}, and obtain 
\begin{equation}\label{Korns-angular}
|\Bu|_{W^{s, p}(D,\mathbb{R}^{2})} \leq |E(\Bu)|_{W^{s,p}(\mathbb{R}^{2}_{x_1 > 0},\mathbb{R}^2)} \leq C [E(\Bu)]_{\mathcal{X}^{s,p}(\mathbb{R}^{2}_{x_1 > 0})} \leq C([\Bu]_{\mathcal{X}^{s,p}(D)} + \|\Bu\|_{L^{p}(D)}).
\end{equation}
}
\end{remark}

\subsection{The case of planar convex polygonal domains}
 In this subsection, we show that extension of vector fields defined in planar convex polygonal domains $D$ to $\bbR^2$ with controlled $\|\cdot\|_{\mathcal{X}^{s,p}}$ norm is possible.
We prove the following extension result. 

{  
\begin{proposition}
Let $n=2$, $s\in (0,1)$ and $1 < p <\infty$. 
Let $\Omega$ be a convex polygonal domain, i.e. $\partial \Omega$ is a simple closed curve that is piecewise affine, with finitely many vertices with interior angle in { $(0,\pi)$.}
Then there exists a positive constant $C$, depending only on $s,p$ and $\Omega,$
    such that for all $\Bu \in W^{s,p}(\Omega;\mathbb{R}^2)$, one has
\begin{equation*}
|{\Bu}|^p_{W^{s,p}(\Omega)} \leq C([{\Bu}]^p_{\mathcal{X}^{s,p}(\Omega)} + \|\Bu\|^p_{L^{p}(\Omega)}).
\end{equation*}
\end{proposition}

\begin{proof}
The proof is similar to that of inequality \eqref{thm:korn2nd}. Choose an open set $\Omega_0 \Subset \Omega$ and open balls $B_{r_j}(\By_j)$, for $j = 1,\ldots,N$ with centers $\By_j$ at the vertices of $\Omega$ such that
\begin{enumerate}
    \item $\Omega = \cup_{j=0}^N \Omega_j$ where $\Omega_j = \Omega \cap B_{r_j}(\By_j)$ for $j = 1,\ldots,N$ and $\By_j\notin B_{r_i}(\By_i)$ if $i\neq j$. 
    \item For every $1\leq j\leq N$, define $T_j: B_{r_{j}}(\By_j) \to \mathbb{R}^{2}$ to be the operator consisting of the translation $\By_j \to 0$ and a rotation such that $T_j(\partial \Omega\cap B_{r_j}(\By_j) )$ coincides with part of an angular planar domain $\{\Bx \in \bbR^2 \, : \, 0 < x_1, \; \alpha_j x_1 < x_2 \}$, for $\alpha_j \in \bbR$. 
\end{enumerate}
Set $Q_j = T_j (B_{r_j}(\By_j)),$ and also define
\[\begin{split}
Q_j^+ := \{ \Bx \in Q_j \, : \, 0 < x_1, \;  \alpha_j x_1 < x_2 \}\,, &\qquad
Q_j^- := \{ \Bx \in Q_j \, : \, 0 < x_1, \;  \alpha_j x_1 > x_2 \}\,, \\
K_j^+ := \{ \Bx \in \mathbb{R}^2 \, : \, \, 0 < x_1, \;  \alpha_j x_1 < x_2 \}\,, &\qquad
K_j^- := \{ \Bx \in \mathbb{R}^2 \, : \, 0 < x_1, \;  \alpha_j x_1 > x_2 \}\,.
\end{split}\]
We may choose the map $T_j$ so that  $T_j(\Omega_j) = Q_{j}^{+}$. Note that $T_j$ is a bi-Lipschitz map with Lipschitz constant depending only on $\Omega$.  Let $\{ \varphi_j \}_{j=0}^N \subset C^{\infty}_c(\mathbb{R}^2;\mathbb{R})$ be a $C^{\infty}$ partition of unity subordinate to the collection $\{\Omega_0\}\cup \{ B_{r_j}(\By_j) \}_{j=1}^N$. Then for every $1\leq j\leq N,$ we have $\text{supp}(\varphi_j) \subset B_{r_j}(\By_j),$  $\text{dist}( \By, \text{supp}(\varphi_j)) \geq \beta_j > 0$ for every $\By \in \Omega \setminus \overline{\Omega_j} $. We also have $\text{supp} (\varphi_0)\subset \Omega_0$, $\text{dist}( \By, \text{supp}(\varphi_0)) \geq \beta_0 > 0$ for every 
$\By \in \Omega \setminus \overline{\Omega_0},$ and $\textstyle \sum_{j=0}^N \varphi_j \equiv 1$ on $\Omega$. 

Suppose now $\Bu\in C^{1}(\overline{\Omega};\mathbb{R}^2)$. Define $\Bu_j := \varphi_j \Bu$, for $j=0, 1, \dots, N.$ Following the exact procedure in the proof of Theorem \ref{Intro:mainthm} we show that 
\begin{equation}\label{est-for-u0-poly}
|\Bu_0|_{W^{s,p}(\Omega)} \leq C \|\Bu\|_{\mathcal{X}^{s,p}(\Omega)} 
\end{equation}
and 
\begin{equation}\label{est-for-uj-by-uj-poly}
|\Bu_j|_{W^{s,p}(\Omega)} \leq 
C \|\Bu\|_{\mathcal{X}^{s,p}(\Omega)}, \qquad j=1, 2, \dots, N,
\end{equation}
where $C$ is a positive constant that depends on $s,p,$ $\text{diam}(\Omega)$, 
and the partition of unity. Therefore by \eqref{est-for-u0-poly} and \eqref{est-for-uj-by-uj-poly}, we have 
\[
[\Bu]_{W^{s,p}(\Omega)} = \Big[\sum_{j=0}^{N} \Bu_j\Big]_{W^{s,p}(\Omega)} \leq \sum_{j=0}^{n}[\Bu_j]_{W^{s,p}(\Omega)} \leq C\|\Bu\|_{\mathcal{X}^{s,p}(\Omega)}.
\]
The estimate for vector fields  $\Bu$ in $W^{s,p}(\Omega;\mathbb{R}^2)$ follows by density. This completes the proof. 
\end{proof}
}

\section*{Acknowledgements}
Davit Harutyunyan's research is supported by NSF 
DMS-2206239. Tadele Mengesha's research is supported by NSF DMS-1910180 and DMS-2206252. James Scott acknowledges support from NSF DMS-1937254 and NSF DMS-2012562. { We are grateful to the anonymous referees who have read the paper very carefully and made suggestions that improved the paper.}

\appendix
\section{Some technical lemmas} 

The following estimate is used in the proof of boundedness of the extension operator in cylindrical epigraphs. The lemma originally appeared in \cite{MS2022} with the restriction that the base function $f$ has a small Lipschitz constant. We prove the lemma without any restriction on $f$ here.

\begin{lemma}
\label{lem:A.1}
Let $f: \mathbb{R}^{n-1}\to \mathbb R$ be Lipschitz with Lipschitz constant $M.$ Let $D\subset\mathbb R^n$ be an epigraph supported by $f$.
For $\lambda>0$ let $\Phi_\lambda(\Bx)\colon  D_{-} \to D$ be as in \eqref{Phi}. Then one has 
$$|\Bx-\By|\leq C|\Phi_\lambda^{-1}(\Bx)-\By| \quad\text{for all}\quad \Bx,\By\in D,$$
for some constant $C=C(\lambda,M).$
\end{lemma}

\begin{proof}
We have $\Phi_\lambda^{-1}(\Bx)=(\Bx', f(\Bx')+\frac{1}{\lambda}(f(\Bx')-x_n)),$ hence we can calculate for any $\Bx,\By\in D_{-}:$
$$
|\Phi_\lambda^{-1}(\Bx)-\By|^2=|\Bx'-\By'|^2+\left|\frac{1}{\lambda}x_n+y_n-(1+\frac{1}{\lambda})f(\Bx')\right|^2.
$$
In the case $x_n\leq y_n$ we have 
\begin{align*}
\frac{1}{\lambda}x_n+y_n-(1+\frac{1}{\lambda})f(\Bx')
&=(y_n-x_n)+(1+\frac{1}{\lambda})(x_n-f(\Bx'))\\
&\geq y_n-x_n\geq 0
\end{align*}
thus we get 
$|\Phi_\lambda^{-1}(\Bx)-\By|\geq|\Bx-\By|.$
Assume in the sequel $x_n> y_n.$ Let $\epsilon=\epsilon(M,\lambda)\in (0,1/2]$ be a small constant yet to be chosen. If $|\Bx'-\By'|\geq \epsilon|\Bx-\By|,$ then we are done.
Assuming further $|\Bx'-\By'|<\epsilon|\Bx-\By|,$ we have 
$$|\Bx'-\By'|<\frac{\epsilon}{\sqrt{1-\epsilon^2}}|x_n-y_n|<2\epsilon|x_n-y_n|.$$
We can then calculate again
\begin{align*}
\frac{1}{\lambda}x_n+y_n-(1+\frac{1}{\lambda})f(\Bx')&=\frac{1}{\lambda}(x_n-y_n)+(1+\frac{1}{\lambda})(y_n-f(\By'))+(1+\frac{1}{\lambda})(f(\By')-f(\Bx'))\\
&\geq \frac{1}{\lambda}(x_n-y_n)-M(1+\frac{1}{\lambda})|\By'-\Bx'|\\
&\geq (\frac{1}{\lambda}-\epsilon M(1+\frac{1}{\lambda}))(x_n-y_n)\\
&=\frac{1}{2\lambda}(x_n-y_n)\\
&\geq 0,
\end{align*} 
if we choose $\epsilon=\mathrm{min}(\frac{1}{2},\frac{1}{2M(1+\lambda)}).$ The proof is now complete.
 \end{proof}

The following estimate is used in the proof of the existence of a bounded extension operator on planar angular domains.  

\begin{lemma}
\label{lem:A.2}
{  Let} 
$f\colon \mathbb{R}^{n-1}\to \mathbb R$ be Lipschitz with Lipschitz constant $M.$ Let $D\subset\mathbb R^n$ be an  epigraph supported by $f$. For the map $\Phi_\lambda(\Bx', x_n) = (\Bx', (1+\lambda)f(\Bx') - \lambda x_n)\colon D_{-}\to D$ with $\lambda>0$, there exists  a constant $C,$ depending only on $n,p,s,\lambda$ and $M,$ such that
\begin{equation*}
I(\Bx)=\int_D \frac{{  |\Bx-\By|^p}}{|\Phi_\lambda^{-1}(\Bx)-\By|^{n+p+ps}} d\By \leq \frac{C}{|x_n-f(\Bx')|^{ps}}\quad\text{for all}\quad \Bx\in D. 
\end{equation*}
 
\end{lemma}

\begin{proof}
For simplicity we will present a proof for $\lambda=1,$ the general case being completely similar. We have 
\begin{align*}
\int_D \frac{{  |\Bx-\By|^p}}{|\Phi_1^{-1}(\Bx)-\By|^{n+p+ps}} d\By&=\int_D \frac{{  (|\Bx'-\By'|^2 + |x_n-y_n|^2)^{p/2}}}{(|\Bx'-\By'|^2+|y_n+x_n-2f(\Bx')|^2)^{(n+p+ps)/2}} d\By.
\end{align*}
For $\epsilon>0$ yet to be chosen, and for any $\Bx\in D$ define the complementary subsets of $D$ as follows: 
\begin{equation*}
E_1^{\epsilon}(\Bx)=\{\By\in D \ : \ |\By'-\Bx'|\geq \epsilon (x_n-f(\Bx'))\},\quad E_2^{\epsilon}(\Bx)=\{\By\in D \ : \ |\By'-\Bx'|< \epsilon(x_n-f(\Bx'))\}.
\end{equation*}
In what follows, the constant $C$ may depend only on $n,p,s,M$ and $\epsilon.$ In the case $\By\in E_1^{\epsilon}(\Bx),$ substitute $a=x_n-f(\Bx')>0,$ $\By'-\Bx'=a\Bw'$ and $y_n+x_n-2f(\Bx')=at,$ where $|\Bw'|\geq \epsilon$ and $\Bw'$ belongs to a subset of $\mathbb R^{n-1}$ and $t$ belongs to a subset of $\mathbb R.$ We can then estimate {using the inequality  $(|\Bw'|^2+|t-2|^2)^{1/2}\leq C (|\Bw'|+|t|+1)$ that }
\begin{align*}
\int_{E_1^{\epsilon}(\Bx)}&\frac{ {  (|\Bx'-\By'|^2 + |x_n-y_n|^2)^{p/2}} }{(|\Bx'-\By'|^2+|y_n+x_n-2f(\Bx')|^2)^{(n+p+ps)/2}}d\By\\ \nonumber
&\leq \frac{1}{a^{ps}} \int_{\mathbb R}
\int_{\{\Bw'\in \mathbb R^{n-1} \ : \ |\Bw'|\geq \epsilon\}}
\frac{ {  (|\Bw'|^2 + |t-2|^2)^{p/2}} }{(|\Bw'|^2+|t|^2)^{(n+p+ps)/2}}d\Bw' dt\\ \nonumber
&\leq \frac{C}{a^{ps}}\int_{\mathbb R}
\int_{\{\Bw'\in \mathbb R^{n-1} \ : \ |\Bw'|\geq \epsilon\}} \frac{1}{(|\Bw'|+|t|)^{n+ps}} {   + \frac{1}{(|\Bw'|+|t|)^{n+ps+p}} }d\Bw' dt\\ \nonumber
&\leq \frac{C}{a^{ps}}\int_{\{\Bw\in \mathbb R^n \ : \ |\Bw|\geq \epsilon\}} \frac{1}{|\Bw|^{n+ps}} {    + \frac{1}{|\Bw|^{n+ps+p}} }d\Bw\\ \nonumber
&=\frac{C(\epsilon)}{a^{ps}},
\end{align*}
Consider now the case $\By\in E_2^{\epsilon}(\Bx).$ We have in that case 
\begin{align*}
y_n+x_n-2f(\Bx')&=(y_n-f(\By'))+(f(\By')-f(\Bx'))+x_n-f(\Bx')\\ \nonumber
&\geq -M|\By'-\Bx'|+x_n-f(\Bx')\\ \nonumber
&\geq (1-\epsilon M)(x_n-f(\Bx')). 
\end{align*}
Thus, if we choose $\epsilon=1/2M$, we will have $y_n+x_n-2f(\Bx')\geq (x_n-f(\Bx'))/2.$ Consequently, setting $\Bx'-\By'=a\Bw'$ and $y_n+x_n-2f(\Bx')=at,$ we will have that $|\Bw'|<\epsilon$ and $\Bw'$ belongs to a subset of $\mathbb R^{n-1},$ while $t$ belongs to a subset of $(1/2,\infty).$
We can estimate in a similar manner:
\begin{align*}
\int_{E_2^{\epsilon}(\Bx)}&\frac{ {  (|\Bx'-\By'|^2 + |x_n-y_n|^2)^{p/2}} }{(|\Bx'-\By'|^2+|y_n+x_n-2f(\Bx')|^2)^{(n+p+ps)/2}}d\By\\ \nonumber
&\leq \frac{1}{a^{ps}}\int_{1/2}^\infty 
\int_{\{\Bw'\in \mathbb R^{n-1} \ : \ |\Bw'|<\epsilon\}} \frac{ {  (|\Bw'|^2 + |t-2|^2)^{p/2}} }{(|\Bw'|^2+|t|^2)^{(n+p+ps)/2}}d\Bw'dt \\ \nonumber
&\leq \frac{1}{a^{ps}}
\int_{1/2}^\infty 
\int_{\{\Bw'\in \mathbb R^{n-1} \ : \ |\Bw'|<\epsilon\}}
\frac{1}{(|\Bw'|+|t|)^{n+ps}} {    + \frac{1}{( |\Bw'| + |t| )^{n+ps+p}} } d\Bw' dt\\ \nonumber
&\leq \frac{C}{a^{ps}}\int_{\{\Bw\in \mathbb R^n \ : \ |\Bw|>1/2\}} \frac{1}{|\Bw|^{n+ps}} {    + \frac{1}{|\Bw|^{n+ps+p}} } d\Bw\\ \nonumber
&\leq \frac{C}{a^{ps}}.
\end{align*}
This completes the proof of the lemma.

\end{proof}


\begin{thebibliography}{999}


\bibitem{NPV} E. Di Nezza, G. Palatucci, and E. Valdinoci. Hitchhiker's guide to the fractional Sobolev spaces, \textit{Bulletin Des Sciences Math\'ematiques,} 
Vol. 136, Issue 5,  pp. 521--573, July--August 2012. 


\bibitem{DMT-compact} Q. Du, T. Mengesha, and X. Tian. Nonlocal criteria for compactness in the space of $L^p$ vector fields, \url{https://arxiv.org/abs/1801.08000}, 2023. 


\bibitem{Dyda} B. Dyda. A fractional order Hardy inequality. Illinois J. Math. 48 (2004), no. 2, 575--588.

\bibitem{HM2023}D. Harutyunyan and H. Mikayelyan. On the fractional Korn inequality in bounded domains: Counterexamples to the case $ps < 1$. To appear in \textit{Advances in Nonlinear Analysis},  https://doi.org/10.1515/anona-2022-0283, 2023.  



\bibitem{Kirzbraun1934} M. D. Kirszbraun.  \"Uber die zusammenziehende und Lipschitzsche Transformationen. \textit{Fundamenta Mathematicae,} 22: 77--108, 1934.


\bibitem{Kond-Oleinik}  V. A. Kondratiev and O. A. Oleinik.  Boundary value problems for a system in elasticity theory in
unbounded domains. Korn inequalities. \textit{Uspekhi Mat. Nauk 43,} 5(263), 55-98, 239,  1988.




\bibitem{Korn1}  A. Korn. Solution g\'en\'erale du probl\`eme d'\'equilibres dans la th\'eorie de l'\'elasticit\'e dans le cas o\`u les efforts
sont donn\'es \`a la surface, \textit{ Ann. Fac. Sci. Toulouse,} ser. 2. 10, 165-269,  1908.


\bibitem{Korn2}  A. Korn. \"Uber einige Ungleichungen, welche in der Theorie der elastischen und elektrischen Schwingungen
eine Rolle spielen, Bull. Int. Cracovie Akademie Umiejet, Classe des Sci. Math. Nat., 705-724, 1909.

\bibitem{Leoni-2023} G. Leoni. A first course in fractional Sobolev spaces. American Mathematical Society, Providence, 2023. 
\bibitem{Loss-Sloane} M. Loss and C. Sloane. Hardy inequalities for fractional integrals on general domains. J. Funct. Anal., 259(6):1369--1379, 2010
\bibitem{Korn-characterization}  T. Mengesha. Nonlocal Korn-type characterization of Sobolev vector fields. \textit{Communications in Contemporary Mathematics}, Volume  14, Number 04, 2012. 
\bibitem{M-half}  T. Mengesha. Fractional Korn and Hardy-type inequalities for vector fields in half space, \textit{Communications in Contemporary Mathematics,} 
Vol. 21, No. 7, 2019. 




\bibitem{Mengesha-Du} T. Mengesha and Q. Du. Nonlocal Constrained Value Problems for a Linear Peridynamic Navier Equation, \textit{Journal of Elasticity,} 
116, 27-51, 2014. 

\bibitem{Mengesha-Du-non} T. Mengesha and Q. Du. 
On the variational limit of a class of nonlocal functionals related to peridynamics. \textit{Nonlinearity,} 28, 3999--4035, 2015.


\bibitem{MS2022}  T. Mengesha and J. M. Scott. A Fractional Korn-type inequality for smooth domains and a regularity estimate for nonlinear nonlocal systems of equations,  \textit{Communications in Mathematical Sciences,} Vol. 20, N0. 2, 405--423, 2022.

\bibitem{Mengesha-Scott-LinLoc} T. Mengesha and J. M. Scott. Linearization and localization of nonconvex functionals motivated by nonlinear peridynamic models. \url{https://arxiv.org/abs/2306.15446}

\bibitem{nitsche1981korn}
J. ~A. Nitsche.
\newblock On {K}orn's second inequality.
\newblock \textit{RAIRO. Analyse num{\'e}rique}, 15(3):237--248, 1981.






\bibitem{Rut.2022} A. Rutkowski. 
\newblock Fractional Korn's inequality on subsets of the Euclidean space.
\newblock \textit{Mathematical Inequalities and Applications,} Vol 25, No.2, 359-367, 2002.





\bibitem{MScott2018Korn}
J. Scott and T. Mengesha.
\newblock A fractional {K}orn-type inequality.
\newblock \textit{Discrete and Continuous Dynamical Systems - A}, 39:3315, 2019.



\bibitem{Silling2001}S.A. Silling.  Reformulation of elasticity theory for discontinuities and long-range forces.  \textit{Journal of the Mechanics and Physics of Solids}, 48, 175--209, 2000.

\bibitem{Silling2010}S.A. Silling. Linearized theory of peridynamic states. \textit{Journal of Elasticity,} 99, 85--111, 2010.
\bibitem{Silling2007} S.A. Silling, M.  Epton, O.  Weckner,  J. Xu, E.  Askari. Peridynamic states and constitutive modeling.  \textit{Journal of Elasticity,} 88, 151--184, 2007.



\bibitem{Stein} E.M. Stein.
\newblock Singular Integrals and Differentiability Properties of Functions, \newblock Princeton Mathematical Series, vol.30, Princeton University Press, Princeton, N.J., 1970.

\bibitem{Temam-Miran} R. Temam and A. Miranville. \newblock {Mathematical Modeling in Continuum Mechanics,} 
\newblock Cambridge University Press, 2001. 





\end{thebibliography}
 \end{document}